\documentclass[12pt, a4paper]{amsart}
\usepackage{amsmath}
\usepackage{amsxtra}
\usepackage{amscd}
\usepackage{amsthm}
\usepackage{amsfonts}
\usepackage{amssymb}
\usepackage {pstricks}
\usepackage{pstricks,pst-node}

\newtheorem{theorem}{Theorem}[section]

\newtheorem{lemma}[theorem]{Lemma}

\newtheorem{corollary}[theorem]{Corollary}
\newtheorem{proposition}[theorem]{Proposition}

\theoremstyle{definition}
\newtheorem{definition}[theorem]{Definition}

\theoremstyle{remark}
\newtheorem{remark}[theorem]{Remark}

\numberwithin{equation}{section}

\newcommand{\lr}{{\longrightarrow}}
\newcommand{\be}{\begin{equation}}
\newcommand{\ee}{\end{equation}}
\newcommand{\inv}{^{-1}}

\newcommand{\wt}{{\rm{wt}}}

\newcommand{\C}{{\mathbb C}}
\newcommand{\Z}{{\mathbb Z}}
\newcommand{\bbS}{{\mathbb S}}

\newcommand{\1}{{\bf 1}}

\newcommand{\id}{{\text{id}}}
\newcommand{\rank}{{\text{rank}}}
\newcommand{\End}{{\text{End}}}
\newcommand{\Hom}{{\text{Hom}}}
\newcommand{\sgn}{{\text{sgn}}}

\newcommand{\cA}{{\mathcal A}}
\newcommand{\cB}{{\mathcal B}}

\newcommand{\cI}{{\mathcal I}}
\newcommand{\cK}{{\mathcal K}}
\newcommand{\cL}{{\mathcal L}}
\newcommand{\cM}{{\mathcal M}}

\newcommand{\cR}{{\mathcal R}}
\newcommand{\cT}{{\mathcal T}}

\newcommand{\cV}{{\mathcal V}}

\newcommand{\fg}{{\mathfrak g}}

\newcommand{\fp}{{\mathfrak p}}
\newcommand{\fl}{{\mathfrak l}}

\newcommand{\gl}{{\mathfrak{gl}}}
\newcommand{\fsl}{{\mathfrak{sl}}}
\newcommand{\so}{{\mathfrak{so}}}
\newcommand{\fsp}{{\mathfrak{sp}}}
\newcommand{\fo}{{\mathfrak{o}}}

\newcommand{\Uq}{{{\rm U}_q}}
\newcommand{\ep}{{\epsilon}}
\newcommand{\ve}{{\varepsilon}}

\begin{document}
\title[Noncommutative classical invariant theory]{A quantum
analogue of the first fundamental theorem of classical
invariant theory}
\author{G. I. Lehrer}
\address{School of Mathematics and Statistics, University of Sydney,
Sydney, Australia} \email{gusl@maths.usyd.edu.au}
\author{Hechun Zhang}
\address{Department of Mathematical Sciences, Tsinghua
University, Beijing, China} \email{hzhang@maths.usyd.edu.au}
\author{R. B. Zhang}
\address{School of Mathematics and Statistics, University of Sydney,
Sydney, Australia} \email{rzhang@maths.usyd.edu.au}

\begin{abstract}
We establish a noncommutative analogue of the first fundamental
theorem of classical invariant theory. For each quantum group
associated with a classical Lie algebra, we construct a
noncommutative associative algebra whose underlying vector
space forms a module for the quantum group and whose algebraic
structure is preserved by the quantum group action. The subspace of
invariants is shown to form a subalgebra, which is finitely
generated. We determine generators of this subalgebra of
invariants and determine their commutation relations. In each case
considered, the noncommutative modules we construct are flat deformations
of their classical commutative analogues. Thus by taking the limit as
$q\rightarrow 1$, our results imply the first fundamental theorem
of classical invariant theory, and therefore generalise them to the
noncommutative case.
\end{abstract}

\maketitle

\tableofcontents

\section{Introduction}

The fundamental theorems of classical invariant theory may be formulated
in three equivalent ways. Given a module $V$ for a reductive group $G$
over a field $k$,
the first formulation provides a complete description of $\End_{kG}(V^{\otimes r})$;
the second does the same thing for the linear space of $G$-invariant multilinear
functions $f:W\to k$, where $W=\oplus^r V \oplus^s V^*$,
i.e. multilinear functions which are constant on $G$-orbits of $W$. Thirdly,
the commutative algebra $S(W^*)$
of polynomial functions on $W$, is naturally a $G$-module, and
one describes the subalgebra of invariant functions. In each case
the `first fundamental theorem' (FFT) provides generators
(suitably defined) for the space of invariants,
while the `second fundamental theorem' (SFT) describes all relations among the
generators. Although equivalent in principle, the statements of the
fundamental theorems in their three formulations are not equally straightforward.
For example, the SFT for orthogonal and symplectic groups in
the first context would require
a hitherto unknown description of an ideal in the Brauer algebras (cf. \cite{LZ2}).

A typical case addressed by the fundamental theorems is when
$G$ is one of the classical groups over the complex numbers $\C$, and
the $G$-module $V$ is the natural module. In these cases, the
standard polynomial (third) form of the first fundamental theorem of
classical invariant theory \cite{W} yields a finite set of algebra
generators for the subalgebra of invariants. In the first formulation
given above, the fundamental theorems are often referred to as
(generalised) Schur-Weyl-Brauer duality.

In this work we shall consider quantum analogues of the FFT in its third formulation.
That is, we shall consider the action of the quantum groups corresponding
to the classical Lie algebras on
non-commutative analogues of the $S(W^*)$ above, and provide a finite set of
generators of the (also generally non-commutative--see below) subalgebra
of invariants. Note that the concept of ``generators'' of course depends 
on the structure of the algebra on which the quantum group acts, 
and that is a crucial aspect of this work.
There is some work in this direction in the
literature. A type of Howe duality between quantum general linear groups was
constructed in \cite{Z03}, which implies the Jimbo-Schur-Weyl
duality between the quantum general linear group and the Hecke
algebra. Also the paper \cite{St} investigated quantum analogues of
polynomial invariants for the symplectic Lie algebra. The works
\cite{GL1,GL2} have elements in common with our work, as does
\cite{BG}. The work \cite{Mo} provides a general context for Hopf algebra
actions on non-commutative algebras. However, to our knowledge, there is no
systematic treatment of the invariant theoretic aspects
of the subject.

There is a considerable literature on endomorphism algebras of tensor powers
of quantum group modules. In particular, when
$\fg$ is a classical Lie algebra and $V$ is the natural module for
the quantum group $\Uq(\fg)$ (a slight enlargement of $\Uq(\fg)$
is more convenient when $\fg=\so_{2n}$), the algebra $\End_{\Uq(\fg)}(V^{\otimes
r})$ at generic $q$ is known to be a quotient of the Hecke algebra
of type $A$ if $\fg=\gl_n$ and a quotient of a Birman-Wenzl-Murakami algebra
in the other cases \cite{J, RW, LR, DPS, LZ1}. The duality between
the quantum general linear group and Hecke algebra is known to be
valid even when $q$ is a root of unity \cite{DPS}. In \cite{LZ1},
the algebras $\End_{\Uq(\fg)}(V^{\otimes r})$,where $V$ is the
$7$-dimensional irreducible module for $\Uq(G_2)$ or any finite
dimensional irreducible module for $\Uq(\fsl_2)$ were described in terms
of representations of the Artin braid group. The endomorphism
algebras for tensor powers of certain irreducible representations of
the $E$ series of quantum groups were studied by Wenzl \cite{Wen}.

The general framework for developing a noncommutative analogue of
classical invariant theory for quantum groups is well established in
the setting of arbitrary Hopf algebras \cite{Mo} (see also \cite{BG}). Let $\Uq(\fg)$ be
a quantum group, and let $A$ be a module algebra over $\Uq(\fg)$ in
the sense of \cite[\S 4.1]{Mo}; that is, $A$ is an associative
algebra whose underlying vector space is a
$\Uq(\fg)$-module, and whose algebraic structure is preserved by the
quantum group action. Then the subspace
$A^{\Uq(\fg)}$ of invariants in $A$ is a subalgebra. Our aim is
to describe the algebraic structure of $A^{\Uq(\fg)}$.

Corresponding to each classical Lie algebra $\fg$, we construct an
associative algebra, which is a module algebra for the quantum group
$\Uq(\fg)$ and reduces in the limit $q\rightarrow 1$ to the
polynomial algebra over a direct sum of copies of the
natural $\fg$-module (and its dual if $\fg$ is type $A$).
Since the usual tensor product of $\Uq(\fg)$ module algebras is
not generally a module algebra, this
construction makes essential use  of the braiding of the quantum
group given by the universal $R$-matrix, which does not appear in the
setting of general Hopf algebras \cite{Mo}. The question of
a module algebra structure on tensor products has also been studied
in \cite{BZ, Zw} and will be discussed below.

We then study the subspace of quantum group invariants in the module
algebra. We shall show that the subspace of invariants always forms a
subalgebra of the module algebra, and find a set of generators
for the subalgebra of invariants together with the commutation relations
obeyed by these generators. This result may be regarded as the FFT of
the noncommutative analogue of classical invariant theory for
quantum groups associated with the classical Lie algebras, in the third of
the three formulations described above.

In all the cases studied, we find
that the subalgebras of invariants are finitely generated. We prove
this by explicit analysis of the subalgebras of invariants using the
diagrammatical method of \cite{RT1, RT2}. We should point out that
since a quantum group $\Uq(\fg)$ is a non-cocommutative Hopf
algebra, the relevant module algebras are noncommutative, and so
also are their subalgebras of invariants. For this reason, most of the
techniques in classical invariant theory,  based on commutative
algebra, do not apply here. In particular the general proofs of finite
generation by Hilbert, Weyl and Nagata in the commutative context do
not generalise to the quantum group setting.

Returning to the problem of constructing appropriate quantum
analogues of polynomial algebras over quantum group modules, we note
that it has long been known that there is no
good quantum analogue of the coordinate ring
of the $4$-dimensional irreducible
$\Uq(\mathfrak{sl}_2)$-module  \cite{Ro}. Recently Berenstein and Zwicknagl
\cite{BZ, Zw} have carried out a systematic study of the braided
symmetric algebras, which might be construed as quantum analogues of
coordinate rings. They found that almost all the braided
symmetric algebras are `smaller' than the corresponding polynomial
algebras thus are not suitable quantum analogues (e.g. in the sense
of \cite[Definition 2]{Ro}) of the latter. That is, they are not flat
deformations of the coordinate rings. The same is true even of
the braided symmetric algebras on direct sums of copies
of the natural module of quantum groups of type $B$, $C$ and $D$.
This may be a partial explanation of the lack of results
available on quantum analogues of classical invariant theory
in the third (polynomial function) formulation of the three described above.
Our construction of quantum analogues of
polynomial algebras will follow an approach suggested in \cite{B},
and the resulting algebras are not braided symmetric algebras in the
sense of \cite{BZ, Zw}.

There is a connection between the present work and
noncommutative geometry \cite{C}. The relationship between 
quantum groups and noncommutative geometry is well described in \cite{M}. 
This suggests that quantum
groups play much the same role in noncommutative geometry as that
played by Lie groups in ordinary differential geometry. Indeed a
noncommutative space $X$ is specified by its associative algebra
$\cA(X)$ of functions  \cite{C}. If $\cA(X)$ is a module algebra
over some quantum group $\Uq(\fg)$, then we may regard $X$ as having
a $\Uq(\fg)$ symmetry algebra. Typical examples of such noncommutative
geometries are the quantum homogeneous spaces studied in \cite{GZ}.
In order to understand the noncommutative space $X$, it is useful to
determine the $\Uq(\fg)$-modules structure of $\cA(X)$, in particular, the
subalgebra of invariants. We shall realise this in Lemma
\ref{odd-dim-sphere} and Corollary \ref{even-dim-sphere}, for 
the quantum sphere.

This paper is organised as follows. Section \ref{general}
discusses generalities concerning module algebras for quantum groups. A
crucial result is Theorem \ref{modulealgebra}, which defines a
twisted multiplication on the tensor product of any two module
algebras by using the universal $R$-matrix of the quantum group to
endow the tensor product with the structure of a module algebra. Sections
\ref{D} and \ref{B} treat the invariant theory of the orthogonal
quantum groups. The main results are summarised in Theorems
\ref{FFT-even} and \ref{FFT-odd}, which provide a noncommutative FFT's
of invariant theory for the quantum orthogonal groups. The
subalgebra of invariants is noncommutative in the quantum setting, and
we describe the commutation relations among the (finite set of)
generators of the subalgebra of invariants in Lemmas
\ref{relations:generators-even} and \ref{relations:generators-odd}.
Section \ref{CA} treats the invariant theory of the symplectic
quantum groups. The noncommutative FFT for the
quantum symplectic groups is Theorem \ref{FFT-sp}, and
the commutation relations among the generators of the noncommutative
subalgebra of invariants are given in Proposition
\ref{relations:generators-C}. In Section \ref{GL}, we study the
invariant theory of the quantum general linear group. 
The techniques used in this section are similar to those of
\cite{GZ} but rather different from those in the previous sections.
Theorem \ref{FFT-gl} provides a noncommutative FFT, and Lemma
\ref{relations:generators-gl} describes the commutation relations
among the generators of the subalgebra of invariants. In
Theorem \ref{skewduality} we prove a quantum analogue of the skew
$(\text{GL}_m, \text{GL}_n)$ duality \cite[Theorem 4.1.1]{H} of
Howe.

Finally, we note that an alternative approach to the questions addressed
here may be to start with the study of endomorphisms of tensor powers
(the first formulation), and translate versions of the first and second fundamental theorems
to the other two settings. One might even hope to apply methods such
as those in the appendix of \cite{ABP} to move from type $A$ to the other classical
types. This strategy seems to provide a route towards a SFT in the quantum setting,
and we intend to return to this theme in a future work.

\section{Generalities on module algebras over quantum
groups}\label{general}
\subsection{Module algebras}
For a simple complex Lie algebra $\fg$, we denote by $\Uq(\fg)$ the
associated quantum group over the field of rational functions
$\cK=\C(q^{\frac{1}{2}})$ where $q$ is an indeterminate. The algebra $\Uq(\fg)$ has a
standard presentation with generators $e_i, f_i, k_i^{\pm 1}$
($1\le i\le \rank(\fg)$) and the usual relations (see, e.g. \cite[Ch. 4]{Ja},
whose conventions we do not follow precisely, but which
is a good general reference). It is well known that
$\Uq(\fg)$ is a Hopf algebra. Write $\Delta$,
$\epsilon$ and $S$ for the co-multiplication, co-unit
and the antipode respectively. We
use the following convention in defining co-multiplication:
\[
\begin{aligned}
\Delta(e_i)&=e_i\otimes k_i + 1\otimes e_i, \\
\Delta(f_i)&=f_i\otimes 1 + k_i^{-1}\otimes f_i, \\
\Delta(k_i)&=k_i\otimes k_i.
\end{aligned}
\]
If $\fg$ is a semi-simple Lie algebra, or a direct sum of general linear algebras, we denote by
$\Uq(\fg)$ the tensor product of the quantum groups of the components of $\fg$.
The definition of the quantum general
linear group $\Uq(\gl_n)$ will be given in Section  \ref{sect-Uqgl}.

An important structural property of a quantum group is its
braiding, namely, the existence of a universal $R$-matrix.
We may think of $R$ as an invertible element in
some appropriate completion of $\Uq(\fg)\otimes \Uq(\fg)$, whose action
on the tensor product of any two finite dimensional $\Uq$-modules
of type $(1, \dots, 1)$ is well defined. It
satisfies the following relations
\begin{eqnarray}
R \Delta(x) = \Delta'(x) R, &\forall x\in \Uq(\fg), \label{R1}\\
(\Delta\otimes\id)R = R_{1 3} R_{2 3}, & (\id\otimes\Delta)R= R_{1
3} R_{1 2}, \label{R2}
\end{eqnarray}
where $\Delta'$ is the opposite co-multiplication. Here the
subscripts of $R_{1 3}$ etc. have the usual meaning as in \cite{D}.
It follows from \eqref{R2} that $R$ satisfies the celebrated
Yang-Baxter equation
\be
R_{1 2} R_{1 3} R_{2 3} = R_{2 3} R_{1 3} R_{1 2}.\label{YB}
\ee

\begin{remark}\label{Rem:R-matrix}
Note that if $V_1,V_2$ are $\Uq$-modules and
$P:V_1\otimes V_2\to V_2\otimes V_1$ interchanges the factors of a tensor,
then in terms of the operator $\check{R}:=PR$, \eqref{R1} reads 
$\check{R}\Delta(x)=\Delta(x) \check{R}$, and
\eqref{YB} becomes the braid relation
$\check{R}_1\check{R}_2\check{R}_1=\check{R}_2\check{R}_1\check{R}_2$.
\end{remark}

Because of the Hopf algebra structure, the tensor product
$V_1\otimes_{\cK} V_2$ of any two $\Uq(\fg)$-modules $V_1$ and $V_2$
becomes a $\Uq(\fg)$-module. In view of \ref{Rem:R-matrix}, the
$R$-matrix (which is invertible) defines an isomorphism of
$\Uq(\fg)$-modules, $V_1\otimes V_2\overset{\check{R}}{\lr}V_2\otimes V_1$.

We shall make extensive use of the notion of  module algebras over a
Hopf algebra \cite[\S 4.1]{Mo}. An
associative algebra $(A, \mu)$ with multiplication $\mu$ and
identity element $1$ is a $\Uq(\fg)$-module algebra if $A$ is a
$\Uq(\fg)$-module such that the $\Uq(\fg)$-action preserves the
algebraic structure of $A$, that is,
\[
x\mu(a\otimes b) = \sum_{(x)} \mu(x_{(1)}(a)\otimes x_{(2)}(b)),
\quad  x(1) = \epsilon(x) 1,
\]
for all $a, b\in A$ and $x\in\Uq(\fg)$.

We shall sometimes use the term $\Uq(\fg)$-algebra as a synonym for
$\Uq(\fg)$-module algebra.

Let $A^{\Uq(\fg)}=\{t\in A \mid x(t)=\epsilon(x) t, \forall
x\in\Uq(\fg) \}$ be the subspace of $\Uq(\fg)$-invariants in the
module algebra $A$. The following result is well known.
\begin{lemma}\label{subalgebra}
Let $A$ be a $\Uq(\fg)$-algebra. Then the subspace
$A^{\Uq(\fg)}$ of invariants is a subalgebra of
$A$.
\end{lemma}
\begin{proof}
Using Sweedler's notation, write the
co-multiplication in $\Uq(\fg)$ as
$\Delta(x)=\sum_{(x)}x_{(1)}\otimes x_{(2)}.$ Then for all $t, t'\in A^{\Uq(\fg)}$,
\[
x(t t') =\sum x_{(1)}(t) x_{(2)}(t') =\sum \epsilon(x_{(1)})
\epsilon(x_{(2)}) t t' = \epsilon(x) t t', \quad \forall
x\in\Uq(\fg).
\]
Thus $tt'\in A^{\Uq(\fg)}$.
\end{proof}
We remark that the above result is valid for arbitrary Hopf algebras.  

A module algebra $A$ is called {\em locally finite} if $\Uq(\fg)a$ is
finite dimensional for any $a\in A$, and is said to be of type $(1,
1, \dots, 1)$ if $A$ is of type $(1, 1,\dots, 1)$ as a 
$\Uq(\fg)$-module. Recall that locally finite $\Uq(\fg)$-modules are
semi-simple (see e.g., \cite{APW}). 
\begin{theorem}\label{modulealgebra}
Let $(A, \mu_A)$ and $(B, \mu_B)$ be locally finite
$\Uq(\fg)$-module algebras of type $(1, 1, \dots, 1)$. Then
$A\otimes B$ acquires a $\Uq(\fg)$-module algebra structure when
endowed with the following multiplication
\begin{eqnarray} \label{multiplication}
\mu_{A, B} = (\mu_A\otimes\mu_B)(\id_A\otimes P R\otimes\id_B),
\end{eqnarray}
where $R$ is the universal $R$-matrix of $\Uq(\fg)$, and $P:
B\otimes A \longrightarrow A\otimes B$ is the permutation map defined by
$a\otimes b\mapsto b\otimes a$.
\end{theorem}
\begin{proof}
This result was stated in \cite{B} and is probably well known.
We give a proof here because it is central to this work. Write the
universal $R$-matrix of $\Uq(\fg)$ as $R=\sum \alpha_t\otimes
\beta_t$, where $\alpha_t, \beta_t$ may be taken to be elements of
$\Uq(\fg)$. Then for $a_i\in A$ and $b_i\in B$ ($i=1, 2, 3$), we
have
\begin{eqnarray*}
&& \left((a_1\otimes b_1) (a_2\otimes b_2)\right) (a_3\otimes b_3)\\
&&= \left(\sum_t a_1 \beta_t(a_2)\otimes \alpha_t(b_1) b_2\right)
(a_3\otimes
b_3)\\
&&=\sum_{t, s} a_1 \beta_t(a_2) \beta_s(a_3) \otimes
\alpha_s(\alpha_t(b_1) b_2)) b_3.
\end{eqnarray*}
Using the first relation of \eqref{R2}, we can re-write the last
expression as
\[
\sum_{t, s, r} a_1 \beta_t(a_2) (\beta_s\beta_r(a_3)) \otimes
(\alpha_s\alpha_t(b_1))  \alpha_r(b_2) b_3.
\]
Then using the second relation of \eqref{R2}, we can re-write
this as
\begin{eqnarray*}
&&\sum_{r, s} a_1 \beta_s(a_2 \beta_r(a_3)) \otimes \alpha_s(b_1)
\alpha_r(b_2) b_3\\
&&= (a_1\otimes b_1) \sum_r a_2 \beta_r(a_3) \otimes \alpha_r(b_2)
b_3\\
&&= (a_1\otimes b_1) \left((a_2\otimes b_2) (a_3\otimes b_3)\right).
\end{eqnarray*}
This proves the associativity of the multiplication.

The action of $\Uq(\fg)$ on $A\otimes B$ is defined in the obvious way:
for all $a\otimes b\in A\otimes B$ and $x\in\Uq(\fg)$,
$x(a\otimes b) := \sum_{(x)} x_{(1)}(a)\otimes x_{(2)}(b)$. Now
\begin{eqnarray*}
x\left((a_1\otimes b_1)(a_2\otimes b_2) \right)&=& x(\sum_t a_1
\beta_t(a_2)\otimes \alpha_t(b_1) b_2)\\
&=& \sum_t \sum_{(x)} x_{(1)}(a_1 \beta_t(a_2))\otimes
x_{(2)}(\alpha_t(b_1) b_2)\\
&=& \sum_t \sum_{(x)} x_{(1)}(a_1) x_{(2)}(\beta_t(a_2))\otimes
x_{(3)}(\alpha_t(b_1)) x_{(4)}(b_2).
\end{eqnarray*}
Using \eqref{R1}, we can re-write this last expression as
\begin{eqnarray*}
&& \sum_t \sum_{(x)} x_{(1)}(a_1) \beta_t(x_{(3)}(a_2))\otimes
\alpha_t(x_{(2)}(b_1)) x_{(4)}(b_2) \\
&=& \sum_{(x)} (x_{(1)}(a_1)\otimes  x_{(2)}(b_1))
(x_{(3)}(a_2)\otimes x_{(4)}(b_2))\\
&=& \sum_{(x)} x_{(1)}(a_1\otimes  b_1) x_{(2)}(a_2\otimes b_2).
\end{eqnarray*}
This proves that the $\Uq(\fg)$ action on $A\otimes B$ preserves the
multiplication.

Next, write $1_A,1_B$ for the identity elements of $A,B$ respectively.
A short calculation using the equations $\sum_t\epsilon(\alpha_t)\beta_t=1$
and $\sum_t\epsilon(\beta_t)\alpha_t=1$ shows that $1_A\otimes 1_B$ is
the identity element of $A\otimes B$. Moreover we have
$x(1\otimes 1) = \epsilon(x) 1\otimes 1$ for all
$x\in\Uq(\fg)$. This completes the proof of the theorem.
\end{proof}

\begin{remark}
Henceforth all $\Uq(\fg)$-modules will be assumed to be locally
finite of type $(1, 1, \dots, 1)$, as will all
$\Uq(\fg)$-module algebras.
\end{remark}

The next two results prove the basic properties of the tensor
product construction.
The first proves the associativity of the tensor product on
the category of $\Uq(\fg)$-algebras.
\begin{lemma}\label{invariants}\label{hyper-associativity}
Let $A$, $B$ and $C$ be three $\Uq(\fg)$-algebras.
The module algebras $(A\otimes B\otimes C, \mu_{A\otimes B, C})$ and
$(A\otimes B\otimes C, \mu_{A, B\otimes C})$ are equal.
\end{lemma}
\begin{proof}
We need to show that the multiplication maps $\mu_{A\otimes B, C}$ and
$\mu_{A, B\otimes  C}$ from $(A\otimes B\otimes
C)^{\otimes 2}$ to $A\otimes B\otimes C$ are identical. For all $a\otimes b\otimes
c$ and $a'\otimes b'\otimes c'$ in $A\otimes B\otimes C$, we have
\[
\begin{aligned}
&\mu_{A\otimes B, C}( (a\otimes b\otimes c)\otimes (a'\otimes
b'\otimes c))\\
&= \sum_t \mu_{A, B}((a\otimes b)\otimes \beta_t(a'\otimes
b'))\otimes \alpha_t(c) c'\\
&= \sum_{r, s, t} a (\beta_r\beta_t(a'))\otimes
\alpha_r(b)\beta_s(b')
\otimes \alpha_s\alpha_t(c) c';\\
&\mu_{A, B\otimes  C}( (a\otimes b\otimes c)\otimes (a'\otimes
b'\otimes c))\\
&= \sum_t a\beta_t(a')\otimes \mu_{B, C}(\alpha_t(b\otimes
c)\otimes (b'\otimes c'))\\
&= \sum_{r, s, t} a (\beta_r\beta_t(a'))\otimes \alpha_r(b)
\beta_s(b') \otimes \alpha_s\alpha_t(c) c'.
\end{aligned}
\]
Hence $\mu_{A\otimes B, C}=\mu_{A, B\otimes  C}$.
\end{proof}

\begin{remark}
Lemma \ref{hyper-associativity} shows that
given $\Uq(\fg)$-algebras $A_i$ ($i=1, 2, \dots, k$), one can
iterate the construction of Theorem \ref{modulealgebra} to obtain an
unambiguous $\Uq(\fg)$-module algebra structure on
$A_1\otimes A_2\otimes\dots\otimes A_k$.
\end{remark}

The next statement sets out the key properties of homomorphisms
of $\Uq(\fg)$-algebras in the tensor category.
\begin{proposition}\label{prop:hom-tensor}
Let $\alpha:A\to A'$ and $\beta:B\to B'$ be homomorphisms of
$\Uq(\fg)$-algebras. Then
\begin{enumerate}
\item The kernel and image of $\alpha$ (and $\beta$) are $\Uq(\fg)$-algebras.
\item The map $\alpha\otimes\beta:A\otimes B\lr A'\otimes B'$ is a homomorphism
of $\Uq(\fg)$-algebras.
\end{enumerate}
\end{proposition}

\begin{proof}
The first assertion is clear.

To verify the second, we need to show that $\alpha\otimes\beta$ respects the
multiplication, and the $\Uq(\fg)$-action. Take elements $a_1,a_2\in A$ and $b_1,b_2\in B$,
and write the $R$-matrix $R$ as $R=\sum_t r_t\otimes s_t$, with $r_t,s_t\in \Uq(\fg)$.
Then $\alpha\otimes\beta(a_1\otimes b_1.a_2\otimes b_2)$
is equal to $\sum_t\alpha(a_1).s_t\alpha(a_2)\otimes r_t\beta(b_1).\beta(b_2)$.
It is straightforward to check that this is equal to the product
$\alpha(a_1)\otimes\beta(b_1).\alpha(a_2)\otimes\beta(b_2)$ in $A'\otimes B'$,
which shows that $\alpha\otimes\beta$ respects algebra multiplication.

Finally, if $u\in \Uq(\fg)$ and $\Delta(u)=\sum_i u_i\otimes u'_i$, then
for elements $a\in A$ and $b\in B$, we have
$\alpha\otimes\beta(u.a\otimes b)=\sum_i\alpha(u_ia)\otimes\beta(u'_ib)
=\sum_iu_i\alpha(a)\otimes u'_i\beta(b)=u.\alpha(a)\otimes\beta(b)$.
Hence $\alpha\otimes\beta$ respects the $\Uq$-action.
\end{proof}

\subsection{Braided symmetric algebras}\label{bsa}

The first step in our construction of the non-commutative 
analogues of coordinate rings with which we shall work is to define,
for each finite dimensional $\Uq(\fg)$-module $V$, a
braided symmetric algebra $S_q(V)$; these algebras have been
considered previously, e.g. in \cite[\S 2.1]{BZ}. 

Write $R_{V V}\in GL(V\otimes V,
\cK)$ for the $R$-matrix associated to $\Uq(\fg)$. Let $P:
V\otimes V\longrightarrow V\otimes V$, be the permutation
map $v\otimes w \mapsto w\otimes
v$, and define $\check R :=P R_{V, V}$. Then
$\check{R}\in \End_{\Uq(\fg)}(V\otimes V)$, and
\begin{eqnarray*}
(\check{R}\otimes \id_V)(\id_V\otimes \check{R})(\check{R}\otimes
\id_V) = (\id_V\otimes \check{R}) (\check{R}\otimes \id_V)
(\id_V\otimes \check{R}).
\end{eqnarray*}
It follows from \cite[Theorem 6.2]{LZ1} that $\check R^2$ acts 
on each isotypic component of $V\otimes V$ as a scalar
of the form $q^m$, where $m\in\Z$.
Hence $\check R$ is semisimple on $V\otimes V$, 
and has characteristic polynomial 
of the form (cf. \cite[(6.10)]{LZ1}) 
\[ \prod_{i=1}^{k_+} \left(t - q^{\chi_i^{(+)}}\right)
\prod_{j=1}^{k_-} \left(t + q^{\chi_i^{(-)}}\right)
\]
where $k_{\pm }$ are positive integers and $\chi^{(+)}_i$ and
$\chi^{(-)}_i$ are rational numbers (in fact $\in\frac{1}{2}\Z$). Define
submodules $a_2$ and $s_2$ of $V\otimes V$ by
\begin{eqnarray}
\begin{aligned}
a_2= \prod_{i=1}^{k_+} \left(\check{R} -
q^{\chi_i^{(+)}}\right)(V\otimes V), \quad  s_2= \prod_{j=1}^{k_-}
\left(\check{R} + q^{\chi_i^{(-)}}\right)(V\otimes V).
\end{aligned}
\end{eqnarray}

Let $T(V)$ be the tensor algebra of $V$. This is a
$\Uq(\fg)$-module in the obvious way, and the $\Uq(\fg)$
action preserves the algebraic structure of $T(V)$; that is, $T(V)$
is a $\Uq(\fg)$ module algebra. Let $\cI_q$ be the two-sided ideal
of $T(V)$ generated by $a_2$; since $\check R$ commutes with $\Uq(\fg)$,
this is stable under
$\Uq(\fg)$. Define the braided symmetric algebra $S_q(V)$ of the
$\Uq(\fg)$-module $V$ by
\[
S_q(V) = T(V)/\cI_q
\]
Let $\tau: T(V)\longrightarrow S_q(V)$ be the natural surjection.
Since the ideal $\cI_q$ is a $\Uq$-subalgebra of $T(V)$, it follows from
Proposition \ref{prop:hom-tensor} that $S_q(V)$ is a
module algebra over $\Uq(\fg)$.

Note that the two-sided ideal $\langle a_2 \rangle$ of $T(V)$
is a $\Z_+$-graded submodule of the $\Z_+$-graded module $T(V)$.
Hence the quotient $S_q(V)$ is $\Z_+$-graded, and we shall denote by
$S_q(V)_n$ its degree $n$ homogeneous component.

The symmetric algebra $S(V)$ of $V$ over $\cK$
is also $\Z_+$-graded. Following \cite{BZ},
$S_q(V)$ is called a flat deformation of $S(V)$ if $\dim S_q(V)_n
=\dim S(V)_n$ for all $n$. This happens if and only if $\dim
S_q(V)_3 =\dim S(V)_3$ by a result of Drinfeld's \cite{D}. In this
case, we shall also say that the $\Uq(\fg)$-module $V$ itself is
flat.

The results of \cite{BZ, Zw} show that flat modules of quantum groups
are extremely rare. Other than the natural modules for the quantum
groups associated with the classical Lie algebras, there are hardly
any modules which are flat. Even in the case of $\Uq(\fsl_2)$,
all the irreducible modules of dimensions greater than $3$ are not
flat!

To develop an analogue of classical invariant
theory for quantum groups, we require quantum analogues of
polynomial algebras on direct sums of copies of a given
module. However even when a module is flat, the direct sum of several
copies of it is usually not flat. For example the natural
module $V$ of $\Uq(\so_n)$, is known to be flat \cite{Zw} (see also
below). Let
$\oplus^m V$ denote the direct sum of $m$ copies of $V$. It is
easily verified that $S_q(\oplus^m V)$ is not flat if $m>1$, and thus is not a 
suitable quantum analogue of $S(\oplus^m V)$.

\subsection{A quantum analogue of the symmetric algebra}
To construct a quantum analogue of $S(\oplus^k V)$, we shall use
Theorem \ref{modulealgebra}.  Given a module $V$ for a quantum
group $\Uq(\fg)$, we have the braided symmetric algebra
$S_q(V)$ as above. Iterating the construction of Theorem
\ref{modulealgebra}, we obtain the $\Uq(\fg)$-module algebra
\begin{eqnarray}\label{Am}
\cA_m(V):=(S_q(V)^{\otimes m}, \mu_S).
\end{eqnarray}
Whenever there is no danger of confusion, we shall drop the $V$ from the
notation and simply write $\cA_m$ for $\cA_m(V)$. This is
a $\Z_+^m$-graded $\Uq$-algebra, with homogeneous component
$(\cA_m)_{\bf d}$ of multi-degree ${\bf d}=(d_1, d_2, \dots, d_m)$
given by
\[ (\cA_m)_{\bf d} = S_q(V)_{d_1}\otimes S_q(V)_{d_2} \otimes \cdots
\otimes S_q(V)_{d_m}.
\]
If $V$ is flat, then $\dim S_q(V)_k= \dim S(V)_k$ for all $k$. In
this case,
\[
\dim \sum_{|{\bf d}|=k} (\cA_m)_{\bf d} = \dim S(\oplus^m V)_k, \quad |{\bf
d}|=\sum_{i=1}^m d_i.
\]
This proves the following statement.
\begin{theorem}
If $V$ is a flat $\Uq(\fg)$-module, the $\Uq(\fg)$-module algebra
$\cA_m(V)$ defined by \eqref{Am} is a flat deformation (in the sense
of \cite{BZ}) of the symmetric algebra $S(\oplus^m V)$ of the direct sum of
$m$-copies of $V$.
\end{theorem}

Let $\{v_s \mid s=1, 2, \dots, \dim V\}$ be a basis for the
$\Uq(\fg)$-module $V$. Since the natural map $\tau: T(V)\longrightarrow S_q(V)$
maps $T(V)_1=V$ injectively into $S_q(V)$, we also denote by $v_s$ the image
of this basis element in $S_q(V)$. We introduce the following
elements
\begin{eqnarray}\label{def:X}
X_{i a}:=\underbrace{1\otimes\dots\otimes 1}_{i-1} \otimes
v_a\otimes \underbrace{ 1\otimes\dots\otimes 1}_{m-i}
\end{eqnarray}
of $\cA_m$. Then for all $i<j$,
\begin{eqnarray}\label{Sm1}
X_{j a} X_{i b} =\sum_{a', b'=1}^{2n} \check{R}_{a' a, b' b}  X_{i
a'}X_{j b'}.
\end{eqnarray}
where $\check{R}_{a a', b b'}$ are the entries of the
$\check{R}$-matrix.

The same construction yields the $\Z_+^m$-graded $\Uq(\fg)$-module algebra
\begin{eqnarray}\label{Tm}
\cT_m(V):=(T(V)^{\otimes m}, \mu_T).
\end{eqnarray}

The next statement is an immediate consequence of Proposition \ref{prop:hom-tensor}.
\begin{lemma}\label{alg-hom}
The natural map $\tau^{\otimes m}: \cT_m\longrightarrow
\cA_m$  is a surjection of $\Z_+^m$-graded $\Uq(\fg)$-algebras.
\end{lemma}

It is useful to describe the map $\tau^{\otimes m}$ explicitly. For
${\bf a}^{(i)}= (a^{(i)}_1, \dots, a^{(i)}_{d_{i}})$ ($i\in[1, m]$),
let $t[{\bf a}^{(i)}] = v_{a^{(i)}_1}\otimes \cdots \otimes
v_{a^{(i)}_{d_i}}\in T(V)_{d_i}$. Write $X[{\bf a}^{(i)}]=X_{i
a^{(i)}_1}\cdots X_{i a^{(i)}_{d_i}}\in S_q(V)_{d_i}$. Then
$\tau^{\otimes m}(t[{\bf a}^{(1)}]\cdots t[{\bf a}^{(m)}])
= X[{\bf a}^{(1)}]\cdots X[{\bf a}^{(m)}]$,
where the right hand side belongs to the homogeneous component of
$\cA_m$ of multi-degree $(d_1, d_2, \dots, d_m)$.

\subsection{The quantum trace and invariants}\label{subsec:canon}
It will be useful to keep
in mind the structural context of the computations which appear below.

Denote $\Uq(\fg)$ by $\Uq$. For any two $\Uq$-modules $N,M$, we have a canonical $\Uq$-module isomorphism
$\xi:N\otimes M^*\overset{\sim}{\lr} \Hom_\cK(M,N)$, where
$\xi(n\otimes \phi):m\mapsto \langle \phi,m\rangle n$, and
the action of $u\in\Uq$ on the right side
is given by $u.f(m)=\sum_{(u)}u_{(1)}.f(S(u_{(2)})m)$ for $f\in \Hom_\cK(M,N)$,
where $\Delta(u)=\sum_{(u)}u_{(1)}\otimes u_{(2)}$. The map $\xi$ induces an isomorphism
\begin{eqnarray}\label{homcanon}
(N\otimes M^*)^{\Uq}\overset{\sim}{\lr} \Hom_\cK(M,N)^{\Uq}=\Hom_{\Uq}(M,N).
\end{eqnarray}

Now taking $M=V^*$ and $N=V$ in \eqref{homcanon}, we obtain an isomorphism
$\xi_0:(V\otimes V^{**})^\Uq\overset{\sim}{\lr} \Hom_\Uq(V^*,V)$.
Write
$K=K_{2\rho}$, where $2\rho$ is the sum of the positive roots (in the
notation of \cite{Ja}). There is an isomorphism of $\Uq$-modules
$\ve_K:V\overset{\sim}{\lr} V^{**}$ where $\ve_K(v)=\ve(Kv)$, and $\ve(w)$ is
evaluation at $w\in V$. Composing these maps, we obtain
an isomorphism
\be\label{vstarv}
f:(V\otimes V)^\Uq\lr (V\otimes V^{**})^\Uq
\lr \Hom_\Uq(V^*,V).
\ee
If $V$ is simple,
the dimension of the right side of \eqref{vstarv} is $0$ or $1$;
in this case any non-zero element $T\in (V\otimes V)^\Uq$ thus gives rise
to an isomorphism of $\Uq$-modules $f_T:V^*\to V$.

Next, note that the map $\tau_1:V^*\otimes V\to \cK$ given by
$\phi\otimes v\mapsto\langle \phi,v  \rangle$ is a homomorphism
of $\Uq$-modules. Applying this with $V$ replaced by $V^*$, we see that
the quantum trace $\tau_q:\End(V)\cong V\otimes V^*\to \cK$
defined by $\tau_q(\xi(v\otimes \phi))= \langle \phi,Kv  \rangle$
is also a $\Uq$-module map.

Now given two $\Uq$-modules $M,N$, we have an isomorphism of $\Uq$-modules
$D:M^*\otimes N^*\overset{\sim}{\lr} (M\otimes N)^*$, defined by $D=\kappa\circ R$,
where $R$ is the $R$-matrix, and $\kappa$ is defined by
$\langle \kappa(\phi\otimes\psi),m\otimes n  \rangle=
\langle \phi,m \rangle\langle \psi, n  \rangle$. This is because
for $u\in\Uq$, $\Delta\circ S(u)=S\otimes S\circ\Delta'(u)$, where
$\Delta'$ is the opposite comultiplication to $\Delta$.

It follows that $\tau_1\in((V^*\otimes V)^*)^\Uq\simeq (V\otimes V^*)^\Uq$
and similarly that
$\tau_q\in((V\otimes V^*)^*)^\Uq\simeq (V^*\otimes V)^\Uq$







\noindent{\bf Invariants in terms of basis elements.}

We shall make explicit the above identifications in terms of bases of $V$ and $V^*$.

Let $V$ be a $\Uq$-module, and let $v_1,\dots,v_n$ be a basis of weight vectors,
with $\wt(v_i)=\lambda_i$; we assume the $\lambda_i$ are pairwise distinct.
Let $v_1^*,\dots,v_n^*$ be the dual basis of $V^*$; then $\wt(v_i^*)=-\lambda_i$.
Under the isomorphism \eqref{homcanon}, $\id_V$ corresponds to
$\gamma:=\sum_iv_i\otimes v_i^*\in V\otimes V^*$. Since the former
is $\Uq$-invariant, we have
\be\label{casimir}
\gamma =\sum_iv_i\otimes v_i^*\in (V\otimes V^*)^\Uq.
\ee

Now the map $\ep_K:V\to V^{**}$ takes $v_r$ to $q^{(2\rho,\lambda_r)}v_r^{**}$
where $(v_r^{**})$ is the basis of $V^{**}$ dual to the basis $(v_r^*)$ of $V^*$.
Applying $\ep_K\otimes\id_{V^*}$ to $\gamma$ and reinterpreting in terms
of $V$ (rather than $V^*$), we see (taking into account that $v_r^*$
has weight $-\lambda_r$) that

\be\label{qcasimir}
\gamma_q =\sum_iq^{-(2\rho,\lambda_i)}v_i^*\otimes v_i\in (V^*\otimes V)^\Uq.
\ee

If $V$ is simple then $(V\otimes V^*)^\Uq\simeq(V^*\otimes V)^\Uq$ is one-dimensional.
Hence the element $\gamma_q$ of \eqref{qcasimir} is a non-zero scalar multiple
of the quantum trace $\tau_q$.

Now assume that $V$ is simple and self dual. Then whenever $\lambda_r$ is a weight
of $V$, so is $-\lambda_r$, and we write $v_{-r}$ for the basis element with weight
$-\lambda_r$. Since any element $T\in (V\otimes V)^\Uq$ has weight zero, it must be
of the form
\be\label{invvxv}
T=\sum_{r=1}^n c_rv_r\otimes v_{-r},
\ee
for elements $c_r\in \cK$. The corresponding $\Uq$-isomorphism $f_T:V^*\to V$ \eqref{vstarv}
is readily seen to satisfy
\be\label{isovstarv}
f_{T}(v_i^*)= c_{-i}q^{(2\rho,\lambda_i)}v_{-i}.
\ee
Since $f_T$ is an isomorphism, it follows that $c_i\neq 0$ for all $i$, and the inverse
$f_T\inv:V\to V^*$ is described similarly.

We may now apply $\id_V\otimes f_T\inv$ and $f_T\inv\otimes \id_V$ to $T$, to obtain invariant
elements of $V\otimes V^*$ and $V^*\otimes V$ respectively. The latter yields
the element $\gamma_q$ of \eqref{qcasimir}, while the former yields
$\sum_ic_ic_{-i}\inv q^{-(2\rho,\lambda_i)}v_i\otimes v_i^*$. Comparing
with $\gamma$ \eqref{qcasimir} shows that
\be
c_ic_{-i}\inv q^{-(2\rho,\lambda_i)} \text{ is independent of } i.
\ee
It is evident that if zero is a weight of $V$, then the implied constant is $1$.

\smallskip

In the next three sections, we shall study the algebras $\cA_m$ associated with the
natural modules of the quantum orthogonal and symplectic groups.

\section{Invariant theory of the quantum even orthogonal groups
 $\Uq(\mathfrak{so}_{2n})$}\label{D}
Let $\epsilon_i$ ($i=1, 2, \dots, n$) be an orthonormal basis of
the weight lattice of the Lie algebra $\so_{2n}$;
a set of simple roots may then be taken to be
$\epsilon_i-\epsilon_{i+1}$  $1\leq i<n$ together with
$\epsilon_{n-1}+\epsilon_n$. In this section,
$\Uq$ will denote the quantum group $\Uq(\mathfrak{so}_{2n})$.

\subsection{The natural module for $\Uq(\mathfrak{so}_{2n})$}\label{sect:naturalmodule}

We realise the group $SO_{2n}(\C)$ as the subgroup of $SL_{2n}(\C)$
preserving the bilinear form defined by the matrix
$J:=\begin{bmatrix}0 & I_n\\ I_n & 0\end{bmatrix}$,
where
$I_n$ is the $n\times n$ identity matrix.
Then
$\fg=\so_{2n}$ is the subalgebra of $\mathfrak{sl}_{2n}(\C)$ consisting of
matrices satisfying $X^t=-JXJ$. Accordingly, there is a
Cartan subalgebra consisting of diagonal matrices. Now the
natural representation of $\so_{2n}$ is minuscule. Thus it lifts to
the natural representation of the quantum group $\Uq$ in such a way
that matrices for the Chevalley generators remain the same
(see, e.g. \cite{ZGB}); this is exploited in the description below.
The natural module $V$ for $\Uq$ has highest weight $\epsilon_1$
and weights $\pm\epsilon_i$, $i=1,2,\dots,n$. It therefore has
a basis $\{v_a \mid
a\in [1, n]\cup[-n, -1]\}$, where $[1, n]=\{1, 2, \dots, n\}$ and
$[-n, -1]$ $={-n},-(n-1),\dots,{-1}$, and $v_a$ has weight $\sgn(a)\epsilon_{|a|}$,
where $\sgn(a)=a/|a|$. Let $E_{a b}$ be the matrix
units in $\End(V)$ relative to this basis, defined by
\begin{eqnarray}\label{Eab}
E_{a b} v_c = \delta_{b c} v_a.
\end{eqnarray}
Then we have the following explicit formulae for the natural
representation  $\pi: \Uq\longrightarrow\End(V)$ of $\Uq$ relative
to the above basis:
\begin{eqnarray*}
\begin{aligned}
&\pi(e_i) = E_{i, i+1} - E_{-i-1, -i},  \qquad  \pi(f_i) =E_{i+1,
i} - E_{-i, -i-1},  \quad i<n, \\
&\pi(e_n) = E_{n-1, -n} - E_{n, -n+1}, \quad  \pi(f_n)=E_{-n, n-1} -
E_{-n+1, n}, \\
&\pi(k_i) = 1 + (q-1)(E_{i i} + E_{-i-1, -i-1}) +
(q^{-1}-1)(E_{i+1, i+1} + E_{-i, -i}),  \quad i<n, \\
&\pi(k_n) = 1 + (q-1)(E_{n-1, n-1} + E_{n n}) + (q^{-1}-1)(E_{-n+1,
-n+1} + E_{-n, -n}).
 \end{aligned}
\end{eqnarray*}

Note that the subalgebra of $\Uq$ generated by $e_i, f_i, k^{\pm
1}_i$ $(i<n)$ is isomorphic to $\Uq(\mathfrak{sl}_n)$,
and the vectors $v_i$ ($i\le n$) span its natural module.
The $v_{-i}$ ($i\le n$) span its dual.

The tensor product $V\otimes V$ is the direct sum of
three distinct irreducible submodules $L_0$, $L_{2\epsilon_1}$ and
$L_{\epsilon_1+\epsilon_2}$ with highest weights $0$, $2\epsilon_1$
and $\epsilon_1+\epsilon_2$ respectively. The following explicit
bases for these irreducible summands will be useful later.

\noindent (1) A basis for $L_0$ (note that this is an element of the form
of $T$ in \eqref{invvxv}):
\begin{eqnarray}\label{basis-trivial}
\sum_{i=1}^n\left(q^{n-i}v_i\otimes v_{-i} + q^{i-n} v_{-i}\otimes
v_i\right).
\end{eqnarray}
(2) A basis for $L_{2\epsilon_1}$:
\begin{eqnarray}\label{basis-Ls}
\begin{aligned}
&v_{i}\otimes v_{i}, \quad v_{-i}\otimes v_{-i}, \quad 1\le i\le n,
\\
&v_i \otimes v_j +q  v_j\otimes v_i, \quad v_{-j} \otimes v_{-i}
+q v_{-i}\otimes v_{-j}, \quad i<j\le n, \\
&v_i \otimes v_{-j} +q  v_{-j}\otimes v_i, \quad  i\ne j, \\
&q^{-1} v_i \otimes v_{-i} +q  v_{-i}\otimes v_i - (v_{i+1}\otimes
v_{-i-1} +v_{-i-1}\otimes v_{i+1}), \quad  i\le n-1.
\end{aligned}
\end{eqnarray}
(3) A basis for $L_{\epsilon_1+\epsilon_2}$:
\begin{eqnarray}\label{basis-La}
\begin{aligned}
&v_i \otimes v_j - q^{-1} v_j\otimes v_i, \quad
v_{-j} \otimes v_{-i} - q^{-1} v_{-i}\otimes v_{-j}, \quad i<j\le n, \\
&v_i \otimes v_{-j} - q^{-1} v_{-j}\otimes v_i, \quad  i\ne j, \\
&v_i \otimes v_{-i} -  v_{-i}\otimes v_i - (q v_{i+1}\otimes
v_{-i-1}
-q^{-1}v_{-i-1}\otimes v_{i+1}), \quad  i< n-1, \\
&v_{n-1} \otimes v_{1-n} -   v_{1-n} \otimes v_{n-1} - (q
v_n\otimes v_{-n} -q^{-1}v_{-n}\otimes v_n), \\
&v_{n-1} \otimes v_{1-n} -   v_{1-n} \otimes v_{n-1} + (q^{-1}
v_n\otimes v_{-n} -qv_{-n}\otimes v_n).
\end{aligned}
\end{eqnarray}

Let $P_s$, $P_a$ and $P_0$ be the idempotent
projections mapping  $V\otimes V$
on to the irreducible submodules with highest weights $2\epsilon_1$,
$\epsilon_1+\epsilon_2$ and $0$ respectively.  Then the $R$-matrix
of $\Uq$ acting on $V\otimes V$ is given by
\begin{eqnarray}
\check{R}= q P_s - q^{-1}
P_a + q^{1-2n}P_0.
\end{eqnarray}

In Section \ref{subsec:qcoordeven}, we shall need the $R$-matrix in the following
slightly more general situation for the proof of Lemma \ref{polynomials-even}.
Let $V_1$ and $V_2$ be two
isomorphic copies of the natural module. Denote by $P_\mu^{(1, 2)}$
the idempotent projection from $V_1\otimes V_2$ onto its irreducible submodule
with highest weight $\mu$, where $\mu = 2\epsilon_1$, $\epsilon_1+\epsilon_2$
or $0$. Similarly, define the idempotents $P_\mu^{(2,1)}$
in $\End(V_2\otimes V_1)$. Then the $R$-matrix
\begin{eqnarray}\label{checkR}
\check{R}: V_1\otimes V_2 \longrightarrow V_2\otimes V_1
\end{eqnarray}
is the (unique) $\Uq$-linear map satisfying the following relations:
\begin{eqnarray}\label{checkR2}
\check{R}P^{(1, 2)}_{2\epsilon_1} = q P^{(2, 1)}_{2\epsilon_1}, &
\check{R}P^{(1, 2)}_{\epsilon_1+\epsilon_1} =- q^{-1} P^{(2,
1)}_{\epsilon_1+\epsilon_2}, &
\check{R}P^{(1, 2)}_0 =q^{1-2n} P^{(2, 1)}_0.
\end{eqnarray}
In particular, if we denote by $\{v_b^{(\alpha)} \mid b\in[-n,
-1]\cup[1,n]\}$ the standard basis for $V_\alpha$ ($\alpha=1, 2$),
and let $T^{(\alpha, \beta)} = \sum_{i=1}^n\left(q^{n-i}v^{(\alpha)}_i\otimes v^{(\beta)}_{-i}
+ q^{i-n} v^{(\alpha)}_{-i}\otimes v^{(\beta)}_i\right)$, we have
\begin{eqnarray}\label{checkR-inv}
\check{R}T^{(1, 2)} = q^{1-2n} T^{(2, 1)}.
\end{eqnarray}

\subsection{The braided symmetric algebra of the natural $\Uq(\so_{2n})$-module}

Let $T(V)$ be the tensor algebra of the natural $\Uq$-module
$V$. Then $T(V)=\bigoplus_{k=0}^\infty V^{\otimes k}$.  The
submodule $P_a(V\otimes V)\cong L_{\epsilon_1+\epsilon_2}$ generates
a graded two-sided ideal $\cI_q$ of $T(V)$ which is also graded since it
is generated by homogeneous elements. Hence the braided symmetric
algebra  $ S_q(V)=T(V)/\cI_q $ of $V$ inherits a $\Z_+$-grading from
$T(V)$. It is evident, given the basis of $L_{\epsilon_1+\epsilon_2}$
constructed above, that the following is a presentation of
$S_q(V)$. Here we abuse notation by writing $v_j$ for the
image in $S_q(V)$ of $v_j\in T(V)$.

\begin{lemma}
The braided symmetric algebra $S_q(V)$ of the natural
$\Uq(\so_{2n})$-module $V$ is generated by $\{v_a\mid a\in[-n,
-1]\cup[1, n]\}$ subject to the following relations
\begin{eqnarray}\label{vv-relations}
\begin{aligned}
&v_i  v_j - q^{-1} v_j v_i=0, \quad 1\leq i<j\le n, \\
&v_{-j}  v_{-i} - q^{-1} v_{-i} v_{-j}=0, \quad 1\leq i<j\le n, \\
&v_i  v_{-j} - q^{-1} v_{-j} v_i=0, \quad  i\ne j, \\
&v_i  v_{-i} -  v_{-i} v_i = q v_{i+1} v_{-i-1}
-q^{-1}v_{-i-1} v_{i+1}, \quad  1\leq i\le n-1, \\
&v_n v_{-n} -v_{-n} v_n=0.
\end{aligned}
\end{eqnarray}
\end{lemma}

Note that the last two sets of relations may be written as
\begin{eqnarray}\label{v+v-}
\begin{aligned}
v_{-i} v_i &=  v_i v_{-i} -(q-q^{-1})q^{i+1-n}\phi^{(+)}_{i+1},
\quad i\in[1, n],
\end{aligned}
\end{eqnarray}
where $ \phi^{(+)}_i$ is the quadratic element
$ \sum_{k=i}^n q^{n-k}v_k v_{-k}$. Set
$\phi^{(-)}_i = \sum_{k=i}^n q^{k-n}v_{-k}v_k. $ An easy
computation using \eqref{v+v-} shows that $\phi^{(+)}_i =
q^{2n-2} \phi^{(-)}_i$ for all $i$.

Using results of \cite{Zw} or by direct calculation, one
sees that the ordered monomials in $v_i$ and $v_{-i}$ ($i\in[1, n]$)
form a basis of $S_q(V)$. That is, $S_q(V)$ is a flat deformation of
the symmetric algebra of $V$ in the sense of \cite{BZ}.

In summary, we have
\begin{theorem}
\begin{enumerate}
\item The braided symmetric algebra $S_q(V)$ of $V$ is a $\Z_+$-graded
      module algebra over $\Uq(\so_{2n})$.
\item The ordered monomials in $v_i$, $v_{-i}$ ($i\in[1, n]$) form a basis of $S_q(V)$.
\end{enumerate}
\end{theorem}

Let
$ \Phi:=\phi^{(+)}_1+\phi^{(-)}_1.$ Then
\begin{eqnarray}\label{Psijj}
\Phi = q^{1-n}(q^{n-1}+ q^{1-n}) \phi_1^{(+)}.
\end{eqnarray}
\begin{proposition}\label{invariants-1}
\begin{enumerate}
\item We have $\Phi\in S_q(V)^{\Uq(\so_{2n})}$.
\item The element $\Phi$ belongs to the centre of $S_q(V)$.
\end{enumerate}
\end{proposition}
\begin{proof}
The first statement is immediate because $\Phi$ is the image in
$S_q(V)$ of the basis element in \eqref{basis-trivial} of $L_0:=(V\otimes V)^\Uq$.

In view of \eqref{Psijj}, to prove (2),
it clearly suffices to show that $\phi^{(+)}_1$ is central.
Consider $v_j \phi^{(+)}_1$: we have
\[
\begin{aligned}
v_j \phi^{(+)}_1 &= \sum_{i=1}^{j-1} q^{n-i} v_i v_{-i} v_j
+ q^{n-j} v_j v_j v_{-j}+ q^{-2} \phi^{(+)}_{j+1}v_j\\
&=\phi^{(+)}_1 v_j + (q^{-2}-1) \phi^{(+)}_{j+1}v_j
+ q^{n-j}v_j(v_j v_{-j} - v_{-j} v_j).
\end{aligned}
\]
By \eqref{v+v-}, we have
\[
q^{n-j}v_j(v_j v_{-j} - v_{-j} v_j)=
(q^2-1) v_j \phi^{(+)}_{j+1} =-(q^{-2}-1) \phi^{(+)}_{j+1}v_j.
\]
It follows that $v_j$ commutes with  $\phi^{(+)}_1$
for all $j\geq 0$. Similarly one shows that $v_{-j}$ commutes with
$\phi^{(+)}_1$ for all $j$, and the proposition follows.
\end{proof}

Let $S_q(V)^{\Uq(\so_{2n})}$ be the space of $\Uq$-invariants in
$S_q(V)$; these form a subalgebra by Lemma \ref{subalgebra}.
\begin{proposition}\label{invariants-2}
The subalgebra $S_q(V)^{\Uq(\so_{2n})}$ of invariants is  generated
by $\Phi$ and is isomorphic to the polynomial algebra $\C[\Phi]$.
\end{proposition}
\begin{proof}
Denote by $S_q(V)_k$ the homogeneous subspace of degree $k$ in
$S_q(V)$. Let $S'_q(V)_k$ be the $\Uq$-submodule of $S_q(V)_k$
generated by
$(v_1)^k$, that is  $S'_q(V)_k= \Uq\cdot (v_1)^k$. Then $S'_q(V)_k$ is
isomorphic to the irreducible $\Uq$-module with highest weight
$k\epsilon_1$, and thus has dimension
$\begin{bmatrix}2n+k-1\\ k\end{bmatrix}-\begin{bmatrix}2n+k-3\\
k-2\end{bmatrix}$.
Taking into account the weights of $\Uq$ occurring in $S_q(V)_{k-2}$, one
sees that $S'_q(V)_k\cap \Phi S_q(V)_{k-2}=0$.
Now $\dim S_q(V)_k = \begin{bmatrix}2n+k-1\\
k\end{bmatrix} = \dim S'_q(V)_k + \dim S_q(V)_{k-2}$, and
recalling that $\Phi$ is invariant, we have the $\Uq$-module decomposition
$S_q(V)_k=S'_q(V)_k\oplus \Phi S_q(V)_{k-2}$.

It follows that $S_q(V)_k^\Uq=S'_q(V)_k^\Uq\oplus (\Phi S_q(V)_{k-2})^\Uq$.
Since $S'_q(V)_k$ is a nontrivial
irreducible $\Uq$-module for all $k> 0$, $S'_q(V)_k^\Uq=0$, and
since $\Phi$ is invariant,
\[\left(S_q(V)_k\right)^{\Uq} = \Phi
\left(S_q(V)_{k-2}\right)^{\Uq}.\] It
follows that $\left(S_q(V)_{k}\right)^{\Uq(\so_{2n})}$ is spanned by
$\Phi^{\frac{k}{2}}$ for even $k$ and is $0$ for odd $k$.
\end{proof}

The results of this section may be applied to the
construction of a quantum sphere $\bbS_q^{2n-1}$ with manifest
quantum orthogonal group symmetry, that is, with an action of
$\Uq(\so_{2n})$. Since $\Phi$ is a central element
in $S_q(V)$ by Proposition \ref{invariants-1} (2), the left ideal
generated by $\Phi-1$ coincides with the right, and hence is a
two-sided ideal, which we
denote by $\langle \Phi-1\rangle$. Then the quantum sphere
$\bbS_q^{2n-1}$ is defined by
\begin{eqnarray}
\bbS_q^{2n-1} =S_q(V)/\langle \Phi-1\rangle.
\end{eqnarray}
{From} the proof of Proposition \ref{invariants-2}, we have the
following result, analogous to its classical ($q=1$) counterpart.
\begin{lemma}\label{odd-dim-sphere}
The quantum sphere $\bbS_q^{2n-1}$ is a $\Uq(\so_{2n})$-module
algebra whose decomposition as $\Uq(\so_{2n})$-module is
$\bbS_q^{2n-1}=\bigoplus_{k=0}^\infty L_{k\epsilon_1}. $
\end{lemma}

\subsection{A $\Uq(\so_{2n})$ module algebra:  $S_q(V)^{\otimes
m}$ with twisted multiplication}\label{subsec:qcoordeven}

Let $\cT_m=(T(V)^{\otimes m}, \mu_T)$ be the $\Uq$-module algebra with
$\Uq=\Uq(\so_{2n})$ where multiplication $\mu_T$ is defined by
iterating \eqref{multiplication}. We shall construct the
quantum analogue of the coordinate ring of $\oplus^m V$ as a
module algebra $\cA_m=(S_q(V)^{\otimes m}, \mu_S)$ over $\Uq$ by
repeatedly using \eqref{multiplication}.

It will be convenient to relabel the standard basis elements of $V$ as
$v_a$ with $a\in[1,2n]$, where $v_a$ is identified with $v_{a-2n-1}$ if $a>n$. The
following result describes the algebraic structure of $\cA_m$.
For $i\in[1, m]$ and $a\in[1, 2n]$, let $X_{i a}$ be the image in
$\cA_m$ of $1\otimes\dots\otimes 1\otimes v_a\otimes 1\otimes\dots\otimes 1$,
where $v_a$ is the $i^{\text{th}}$ factor.

\begin{lemma}\label{polynomials-even}
The $\Uq(\so_{2n})$-module algebra $\cA_m$ is generated by $X_{i a}$
with $i\in[1, m]$ and $a\in[1, 2n]$, subject to the following
relations:

(1) for fixed $i$, the elements $X_{i a}$ obey the relations
\eqref{vv-relations} with $v_a=X_{i a}$;

(2) for $i<j$ in $[1, m]$ and $a, b\in[1, 2n]$:
\begin{eqnarray}
\begin{aligned}\label{Sm2}
X_{j a} X_{i a} &=q  X_{i a} X_{j a}, \quad \forall a,\\
X_{j b} X_{i a} &= X_{i a} X_{j b}+(q-q^{-1}) X_{i b} X_{j a}, \quad a<b\ne 2n+1-a,\\
X_{j a} X_{i b} &=  X_{i b} X_{j a}, \quad a<b\ne2n+1-a,
\end{aligned}\\
\begin{aligned}\label{Sm3}
X_{j t} X_{i, 2n+1-t} &=  q  X_{i, 2n+1-t} X_{j t} - (q-q^{-1}) q^{n-t} \psi_t^{(i, j)}, \\
X_{j, 2n+1-t} X_{i t}  &=  q X_{i t} X_{j, 2n+1-t}  -  (q-q^{-1}) X_{i, 2n+1-t} X_{j t}\\
&+(q-q^{-1})q^{t-n}\left(\bar\psi_{t+1}^{(i, j)}-\Psi^{(i,j)}\right), \quad t\in[1, n],
\end{aligned}
\end{eqnarray}
where (for all $i, j\in[1, m]$)
\[
\begin{aligned}
\psi_t^{(i, j)} &=\sum_{k=1}^t q^{k-n} X_{i, 2n+1-k} X_{j k},\\
\bar\psi_t^{(i, j)} &=\sum_{k=1}^t q^{n-k} X_{i k} X_{j, 2n+1-k},\\
\Psi^{(i, j)}&= \psi_n^{(i, j)}+\bar\psi_n^{(i, j)}.
\end{aligned}
\]
\end{lemma}
\begin{proof}
For fixed $i$, the elements $X_{i a}$ ($i\in[1,2n]$) generate a
subalgebra isomorphic to $S_q(V)$ in $\cA_m$, whence (1).

Part (2) is obtained from \eqref{Sm1} by
straightforward but tedious calculation. We indicate the main
steps. First, using \eqref{checkR2}, we rewrite
relation \eqref{Sm1} for the present case in the following form (for all $i<j$):
\begin{eqnarray}
\begin{aligned}\label{symm}
&X_{j a} X_{i a} = q X_{i a} X_{j a}, \ \forall a, \\
&X_{j a} X_{i b} + q X_{j b} X_{i a} \\
&= q \left(X_{i a} X_{j b} + q
 X_{i b} X_{j a} \right), \quad  a<b\ne 2n+1-a,\\
&q^{-1} X_{j t} X_{i, 2n+1-t} + q X_{j, 2n+1-t} X_{i t} \\
& - \left(X_{j, t+1} X_{i, 2n-t} + X_{j, 2n-t} X_{i, t}\right)\\
&= q\left(q^{-1} X_{i t} X_{j, 2n+1-t} + q X_{i, 2n+1-t} X_{j t}
\right)\\
&-q\left(X_{i, t+1} X_{j, 2n-t} + X_{i, 2n-t} X_{j, t}\right), \quad
1\le t<n;
\end{aligned}
\end{eqnarray}
\begin{eqnarray}
\begin{aligned}\label{skew}
&X_{j a} X_{i b}- q^{-1} X_{j b} X_{i a} \\
&= -q^{-1} \left(X_{i a} X_{j b} -q^{-1}
 X_{i b} X_{j a} \right), \quad  a<b\ne 2n+1-a,\\
&X_{j t} X_{i, 2n+1-t} - X_{j, 2n+1-t} X_{i t} \\
& - \left(qX_{j, t+1} X_{i, 2n-t} -q^{-1} X_{j, 2n-t} X_{i, t}\right)\\
&= -q^{-1}\left(X_{i t} X_{j, 2n+1-t} - X_{i, 2n+1-t} X_{j t}
\right)\\
&+q^{-1}\left(qX_{i, t+1} X_{j, 2n-t} -q^{-1} X_{i, 2n-t} X_{j,
t}\right), \quad 1\le t<n,\\
&X_{j n} X_{i, n+1} - X_{j, n+1} X_{i n} \\
&=-q^{-1}\left(X_{i n} X_{j, n+1} - X_{i, n+1} X_{j n} \right);
\end{aligned}
\end{eqnarray}
\begin{eqnarray}
\Psi^{(j, i)} = q^{1-2n} \Psi^{(i, j)}.\label{trace}
\end{eqnarray}

The first relation of \eqref{Sm2} is just the first relation of
\eqref{symm}, and the other relations of \eqref{Sm2} are 
obtained by combining the second relation of \eqref{symm} with the
first relation in \eqref{skew} (that is, the relations with $a<b\ne
2n+1-a$).

The relations \eqref{Sm3} are obtained from the third relation of
\eqref{symm}, the second and third relations of \eqref{skew} and the
relation \eqref{trace} by a rather lengthy sequence of routine manipulations.

The fact that these relations suffice is a consequence of the flat
nature of $S_q(V)$, which implies that there is a linear isomorphism
$\cA_m\lr S(V)^{\otimes m}$.
\end{proof}

The elements $\Psi^{(i, j)}$ will play an important role in the
study of the subalgebra of $\Uq$-invariants in $\cA_m$, and we
now study their properties.
\begin{lemma}\label{PsiX}
\begin{enumerate}
\item
For all $i, j\in[1, m]$, the elements $\Psi^{(i, j)}$ are
$\Uq$-invariant.
\item The elements $\Psi^{(i, j)}$ satisfy \eqref{trace} as well as the following relations:
\begin{eqnarray}\label{XPsi}
\begin{aligned}
&X_{k a} \Psi^{(i, i)}-  \Psi^{(i, i)}X_{k a}=0, \quad \text{for all \ } i, k, \\
&X_{k a} \Psi^{(i, j)}-  \Psi^{(i, j)}X_{k a}=0, \quad k> i, j \ \ \text{or} \ k<i, j, \\
&X_{k a} \Psi^{(i, j)}-  \Psi^{(i, j)}X_{k a}\\
&=(q-q^{-1}) \left(X_{i a} \Psi^{(k, j)}- \Psi^{(i, k)}X_{j a}\right), \quad  i<k<j,\\
&\Psi^{(i, j)}X_{i a} - q^{-1}  X_{i a} \Psi^{(i, j)}
= (q-q^{-1})\bar\psi_n^{(i, i)} X_{j a}, \quad i<j, \\
&X_{j a} \Psi^{(i, j)} - q^{-1} \Psi^{(i, j)}X_{j a} = (q-q^{-1})
 X_{i a} \bar\psi_n^{(j, j)}, \quad i<j.
\end{aligned}
\end{eqnarray}
\end{enumerate}
\end{lemma}
\begin{proof}
Observe that the image of $V$ under the natural map
$T(V)\to S_q(V)$ is a copy of $V$ in $S(V)$. Hence
for $i<j$, $\Psi^{(i,j)}$ may be thought of as
the basis element \eqref{basis-trivial} of $(V\otimes V)^\Uq$
where the two copies of $V$ are in the $i^{\text th}$ and $j^{\text th}$
factors of $\otimes^mS_q(V)$.
Taking Proposition \ref{invariants-1} into account,
it follows that $\Psi^{(i, j)}$ is invariant if $i\le j$.
By equation \eqref{trace}, (1) follows.

The $k=i$ case of the first relation  of \eqref{XPsi}  follows from
Proposition \ref{invariants-1}.  Write the universal $R$-matrix of $\Uq$
as $R=\sum_t\alpha_t\otimes\beta_t$. If $k>i, j$, we have
\[
X_{k a} \Psi^{(i, j)} = \sum_t \beta_t(\Psi^{(i, j)}) \alpha_t(X_{k a}).
\]
Since $\Psi^{(i, j)}$ is $\Uq$-invariant by part (1), the right side
is equal to $$\sum_t \ep(\beta_t)(\Psi^{(i, j)}) \alpha_t(X_{k a})
=\Psi^{(i, j)}X_{k a},$$ proving the second
relation of \eqref{XPsi} for $k>i, j$. The case $k<i, j$
is similar. Taking $i=j$, we obtain the case $k\ne i$ of
the first relation of \eqref{XPsi}.

To prove the third relation of \eqref{XPsi}, it suffices to consider
the case $i=1$, $k=2$ and $j=3$.  Since $S_q(V)_1\cong T(V)_1\cong
V$, we have canonical $\Uq$-module isomorphisms \[ V\otimes V\otimes V
\longrightarrow\left(\cT_3(V)\right)_{(1, 1, 1)}\longrightarrow
\left(\cA_3\right)_{(1, 1, 1)},\] where the second map is the
restriction of $\tau^{\otimes 3}$ to $\left(\cT_3(V)\right)_{(1, 1,
1)}$. Denote the first map by $\iota_{(1, 1, 1)}$. Our strategy is
to deduce the third relation of \eqref{XPsi} from appropriate
relations in $V\otimes V\otimes V$.

This will be done using the diagrammatical method of
Reshetikhin-Turaev (see \cite{RT1, RT2}), which is equivalent
to working in the $BMW$-algebra, to describe homomorphisms between
tensor powers of $V$. Recall that their
functor sends tangle diagrams to $\Uq$-module homomorphisms. In our
case, we shall colour all the components of any tangle diagram with
the module $V$, so that the images of tangle diagrams under the
Reshetikhin-Turaev functor are $\Uq$-maps between tensor powers of
$V$. Because the module $V$ is self dual, there is no need to
orient the tangle diagrams.

We shall identify tangle diagrams with the
corresponding $\Uq$-module homomorphisms;
the relations we seek will
arise from relations among diagrams whose images lie in $V^{\otimes 3}$.
Write the basis
element \eqref{basis-trivial} of $L_0\subset V\otimes V$ defined
in Section \ref{sect:naturalmodule} as $\sum_{a, b} C_{a b}
v_a\otimes v_b$ with $C_{a b}\in\cK$; then in terms of diagrams,
we have \begin{picture}(15, 10) \qbezier(0,10)(5, -15)(10,10)
\end{picture}: $\cK\longrightarrow V\otimes V$, $1\mapsto
\sum_{a, b} C_{a b} v_a\otimes v_b$.  Furthermore, we have the
following skein relation:
\begin{eqnarray}\label{skein}
\text{\begin{picture}(15, 10) \qbezier(-2,0)(-10, 0)(-10,10)
\qbezier(2,0)(10, 0)(10,10) \put(0, -5){\line(0, 1){15}}
\end{picture}}-
\text{\hspace{.4cm}\begin{picture}(15, 10) \qbezier(0,0)(-10,
0)(-10,10) \qbezier(0,0)(10, 0)(10,10) \put(0, -5){\line(0, 1){4}}
\put(0, 1){\line(0, 1){9}}
\end{picture}}  = (q-q^{-1})\left( \text{\hspace{.2cm}\begin{picture}(10, 10)
\qbezier(0,10)(5, -15)(10,10) \put(-3, -5){\line(0, 1){15}}
\end{picture}}
- \text{\begin{picture}(10, 10) \qbezier(0,10)(5, -15)(10,10)
\put(13, -5){\line(0, 1){15}}
\end{picture} \hspace{.15cm}}  \right).
\end{eqnarray}
Let us denote
\begin{eqnarray*}
\begin{aligned}
D_+:=\text{\hspace{.2cm}
\begin{picture}(10, 10)
\qbezier(-2,0)(-10, 0)(-10,10) \qbezier(2,0)(10, 0)(10,10) \put(0,
-5){\line(0, 1){15}}
\end{picture}\ :  $V \longrightarrow V^{\otimes 3}$}, &\quad&
D_-:=\text{\hspace{.2cm}
\begin{picture}(10, 10)
\qbezier(0,0)(-10, 0)(-10,10) \qbezier(0,0)(10, 0)(10,10) \put(0,
-6){\line(0, 1){4}} \put(0, 2){\line(0, 1){9}}
\end{picture}\ : $V \longrightarrow V^{\otimes 3}$}, \\
D_0:= \text{\hspace{.2cm}
\begin{picture}(15, 10)
\qbezier(0,10)(5, -15)(10,10) \put(-3, -5){\line(0, 1){15}}
\end{picture}\ : $V \longrightarrow V^{\otimes 3}$}, &\quad&
D_0':=\text{\hspace{.2cm}
\begin{picture}(15, 10)
\qbezier(0,10)(5, -15)(10,10) \put(13, -5){\line(0, 1){15}}
\end{picture}\ : $ V \longrightarrow V^{\otimes 3}$}.
\end{aligned}
\end{eqnarray*}
Then 
\begin{eqnarray*}
\begin{aligned}
&D_0(v_a) = v_a\otimes \sum_{b, d} C_{b d} v_b\otimes v_d, &\quad&
D'_0(v_a) = \sum_{b, d} C_{b d} v_b\otimes v_d \otimes v_a,\\
&D_+(v_a) = \sum_{b, d} C_{b d} \check{R}(v_a\otimes  v_b)\otimes
v_d, &\quad& D_-(v_a) = \sum_{b, d} C_{b d} v_b\otimes \check{R}(v_d
\otimes v_a),
\end{aligned}
\end{eqnarray*}
from which we obtain
\begin{eqnarray*}
\tau^{\otimes 3}\circ \iota_{(1, 1, 1)}\circ D_+(v_a) = X_{2 a}
\Psi^{(1, 3)}, &\quad&
\tau^{\otimes 3}\circ \iota_{(1, 1, 1)}\circ D_-(v_a) = \Psi^{(1, 3)} X_{2 a}, \\
\tau^{\otimes 3}\circ \iota_{(1, 1, 1)}\circ D_0(v_a) = X_{1 a}
\Psi^{(2, 3)}, &\quad& \tau^{\otimes 3}\circ \iota_{(1, 1, 1)}\circ
D'_0(v_a) = \Psi^{(1, 2)}X_{3 a}.
\end{eqnarray*}
Evaluating both sides of \eqref{skein} at $v_a$, then applying the
map $\tau^{\otimes 3}\circ \iota_{(1, 1, 1)}$ to the resulting
elements of $V^{\otimes 3}$, we obtain the third relation of
\eqref{XPsi} for $i=1$, $k=2$ and $j=3$.

The fourth and fifth relations of \eqref{XPsi} can be proved in much
the same way, and we shall consider the fourth relation only. We
may assume that $i=1$ and $j=2$. Note that the map
$\tau^{\otimes 2}: \cT_2(V)\longrightarrow \cA_2$ restricts to an
isomorphism $(P_s+P_0)T(V)_2\otimes T(V)_1\cong S_q(V)_2\otimes
S_q(V)_1$. Denote by $\iota_{(2, 1)}$ the canonical $\Uq$-isomorphism from
$(V\otimes V)\otimes V$ to $T(V)_2\otimes T(V)_1$. Let
\begin{eqnarray*}
F:=((P_s+P_0)\otimes\id)D_0 &= (P_s\otimes\id)D_0 + \frac{1}{\dim_q V}D_0', \\
B:=((P_s+P_0)\otimes\id)D_- &= q^{-1}(P_s\otimes\id)D_0 +
\frac{q^{2n-1}}{\dim_qV}D_0',
\end{eqnarray*}
where $\dim_qV=[n]_q (q^{n-1}+q^{1-n})$ is the quantum dimension of
$V$. Then
\[
F - q B = -q^n \frac{q^{n}-q^{-n}}{\dim_q V} D'_0:  V^{\otimes 3}
\longrightarrow V^{\otimes 3}.
\]
Note that
\begin{eqnarray*}
\tau^{\otimes 2}\circ \iota_{(2, 1)}\circ F(v_a)=X_{1 a} \Psi^{(1,
2)}, &\quad& \tau^{\otimes 2}\circ \iota_{(2, 1)}\circ
B(v_a)=\Psi^{(1, 2)}X_{1 a}, \\
{\rm and }\;\tau^{\otimes 2}\circ \iota_{(2, 1)}\circ D'_0(v_a)=\Psi^{(1,
1)}X_{2 a}.
\end{eqnarray*}
Thus for $i=1$ and $j=2$, we have
\[
q \Psi^{(i, j)} X_{i a} - X_{i a} \Psi^{(i, j)}  = q^n
\frac{q^{n}-q^{-n}}{\dim_q V} \Psi^{(i, i)} X_{j a}.
\]
This leads to the desired result taking into account that
\begin{eqnarray}
\Psi^{(i, i)}=\frac{q^{1-n} \dim_q V}{[n]_q} \bar\psi_n^{(i, i)},
\end{eqnarray}
which is implied by \eqref{Psijj}. This completes the proof of the Lemma.
\end{proof}
The following relations are easy consequences of part (2) of Lemma
\ref{PsiX}.
\begin{lemma}\label{relations:generators-even}
\begin{eqnarray}
\begin{aligned}
&\Psi^{(i, i)}\Psi^{(j, k)}-\Psi^{(j, k)}\Psi^{(i, i)}=0,
\quad \text{for all \ } \ i, j, k,\\
&\Psi^{(i, k)}\Psi^{(i, j)}-q^{-1}\Psi^{(i, j)}\Psi^{(i, k)}=
(q-q^{-1})\bar\psi_n^{(i, i)}\Psi^{(j, k)}, \quad k\ne i, j; \ i<j, \\
&\Psi^{(j, k)}\Psi^{(i, j)}-q^{-1}\Psi^{(i, j)}\Psi^{(j, k)}=
(q-q^{-1})\bar\psi_n^{(j, j)}\Psi^{(i, k)}, \quad k\ne i, j; \ i<j,\\
&\Psi^{(i, j)}\Psi^{(k, l)}-\Psi^{(k, l)}\Psi^{(i, j)}=0,
\quad k< i<j< l,\\
&\Psi^{(i, j)}\Psi^{(k, l)}-\Psi^{(k, l)}\Psi^{(i, j)}\\
&=(q-q^{-1})
\left(\Psi^{(i, k)}\Psi^{(j, l)} -\Psi^{(i, l)}\Psi^{(k, j)} \right),
\quad i<k<j< l.
\end{aligned}
\end{eqnarray}
\end{lemma}

\subsubsection{The algebra $\Uq({\mathfrak{o}}_{N})$}\label{On}

Following \cite{LZ1} we introduce extra generators $\sigma^{\pm 1}$ to
augment $\Uq(\so_N)$, obtaining a new algebra, which we denote by
$\Uq({\mathfrak{o}}_N)$. Here $\sigma^{\pm 1}$ are mutual inverses,
and have the following properties: If $N$ is odd, $\sigma$
commutes with all the generators of $\Uq(\so_N)$. When $N=2n$, we
label the last two simple roots of $\fg$ as
$\alpha_{n-1}=\epsilon_{n-1}-\epsilon_n$ and
$\alpha_n=\epsilon_{n-1}+\epsilon_n$. Then
\begin{eqnarray*} \sigma e_{n-1}\sigma^{-1} = e_{n},&&
\sigma e_{n}\sigma^{-1} = e_{n-1},\\
\sigma f_{n-1}\sigma^{-1} = f_{n},&&
\sigma f_{n}\sigma^{-1} = f_{n-1},\\
\sigma k_{n-1}\sigma^{-1} = k_{n},&& \sigma k_{n}\sigma^{-1} =
k_{n-1},\end{eqnarray*} while all the other generators commute with
$\sigma$. Denote by $\Uq(\mathfrak{o}_N)$ the associative algebra
generated by $\Uq(\so_N)$ and $\sigma^{\pm 1}$. We extend the
co-multiplication and antipode of $\Uq(\so_N)$ to
$\Uq(\mathfrak{o}_N)$ by letting
$\Delta(\sigma)=\sigma\otimes\sigma$, and $S(\sigma)=\sigma^{-1}$.
Then $\Uq(\mathfrak{o}_N)$ acquires the structure of a Hopf algebra.
Now $\sigma$ acts as an automorphism on $\Uq(\so_N)$ by conjugation.
The $R$-matrix of $\Uq(\so_N)$ is invariant under the automorphism,
i.e., $\Delta(\sigma)$ commutes with the $R$-matrix.

We extend the natural $\Uq(\so_N)$-module $V$ to a
$\Uq(\mathfrak{o}_N)$-module by stipulating that
\begin{enumerate}
\item[(i)] For odd $n$, $\sigma$ acts on the highest weight vector of $V$ by $-1$;
\item[(ii)] For even $n$, $\sigma$ acts on the highest weight vector of $V$ by $1$.
\end{enumerate}
Then $\Uq(\mathfrak{o}_N)$ also acts on tensor powers of $V$ through
the co-multiplication. We shall consider the $\Uq(\mathfrak{o}_N)$-modules
which are direct sums of submodules of $V^{\otimes r}$, $r\ge 0$.

Define $\Uq$-endomorphisms $b_i = \id_V^{\otimes (i-1)}\otimes
\check{R}\otimes \id_V^{\otimes (r-1)}$  ($1\le i\le r-1$) of
$V^{\otimes r}$.  These define a representation of the
Birman-Wenzl-Murakami (BMW) algebra, and we denote by $\cB(r)$ the
image of the BMW algebra in this representation. Then we have the following
result (see, e.g., \cite{LZ1}).

\begin{theorem}
\label{FT-O} If $V$ is the natural module for the quantum orthogonal
group $\Uq(\mathfrak{o}_N)$, for any integer $r\ge 2$,
\[\End_{\Uq(\mathfrak{o}_N)}(V^{\otimes r})=\cB(r).\]
\end{theorem}

\subsection{Noncommutative FFT of invariant theory}

The following result is a quantum analogue of the third
(coordinate ring) formulation of the first fundamental theorem for the invariant
theory of ${\mathfrak o}_{2n}$.
\begin{theorem}\label{FFT-even}\label{o-FFT}
The subalgebra $\cA_m^{\Uq(\fo_{2n})}:= \{ f\in \cA_m \mid x(f) =
\epsilon(x) f, \ \forall x\in \Uq(\fo_{2n})\}$ of
$\Uq(\fo_{2n})$-invariants in $\cA_m$ is generated by the elements
$\Psi^{(i, j)}$ ($i\le j$) and the identity.
\end{theorem}
\begin{proof} Let $A=\cA_m^{\Uq(\fo_{2n})}$ and
$B= \cT_m^{\Uq(\fo_{2n})}$, where $\cT_m^{\Uq(\fo_{2n})}$ is the
subalgebra of $\Uq(\fo_{2n})$-invariants of $\cT_m$. Since both
$\cT_m$ and $\cA_m$ are semi-simple as $\Uq$-modules, the surjection
$\tau^{\otimes m}$ descends to $B$, giving rise to a surjective
algebra homomorphism $B\longrightarrow A$.

Let ${\bf d}=(d_1, d_2, \dots, d_m)$, and set
\[
(\cT_m)_{\bf d} = T(V)_{d_1}\otimes \cdots \otimes T(V)_{d_m},
\quad (\cA_m)_{\bf d} = S_q(V)_{d_1}\otimes \cdots \otimes S_q(V)_{d_m}.
\]
Then the homogeneous component $A_{\bf d}$ of $A$ of multi-degree ${\bf d}$ is
\[
A_{\bf d} = ((\cA_m)_{\bf d})^{\Uq} = \tau^{\otimes m}
\left(((\cT_m)_{\bf d})^{\Uq}\right)=\tau^{\otimes m}(B_{\bf d}),
\]
where $B_{\bf d}$ is the homogeneous subspace of $B$  of
multi-degree ${\bf d}$. By considering the weights of $(\cT_m)_{\bf
d}$, one sees that $B_{\bf d}=0$ if $|d|=\sum d_i$ is odd.
For $|d|=2k$, we consider the $(0, 2k)$-tangle diagrams. The image
of each diagram under the Reshetikhin-Turaev functor is an element
of $\Hom_{\Uq(\fo_{2n})}(\C(q), V^{\otimes 2k})$.  Since
(cf. \S \ref{subsec:canon})
\[
\Hom_{\Uq(\fo_{2n})}(\C(q), V^{\otimes 2k}) \cong
\End_{\Uq(\fo_{2n})}(V^{\otimes k}) \quad \text{as vector space},
\]
the images of the $(0, 2k)$-tangle diagrams under the
Reshetikhin-Turaev functor span $\Hom_{\Uq(\fo_{2n})}(\C(q),
V^{\otimes 2k})$ by Theorem \ref{FT-O}.

Note that there is a canonical isomorphism
$\Hom_{\Uq(\fo_{2n})}(\C(q), V^{\otimes 2k})
\stackrel{\sim}{\longrightarrow} B_{\bf d}$,
given by $\phi\mapsto \phi(1)$.
Thus we shall regard elements of $B$ as linear combinations
of $(0, 2r)$ tangle diagrams for $r\ge 0$. Two tangle diagrams in
$B$ may be multiplied in a natural way with the multiplication being
that inherited from $\cT_m$. We shall use this algebraic structure
of $B$  presently.

\begin{remark}
It is important to note that this multiplication is
different from the standard BMW-type multiplication obtained by composing
diagrams, which has also arisen in this work. It is the quantum analogue
of the commutative (pointwise)
multiplication of coordinate functions in the classical case (i.e. setting three
of the introduction), and changes degree. In terms of
diagrams, it corresponds to setting diagrams side by side, rather
than concatenating them.
\end{remark}

Because the $\check{R}$-matrix in the natural $\Uq$-representation
satisfies Kaufman's skein relation, all $(0, 2k)$-tangle diagrams
can be expressed as linear combinations of diagrams with the
property that any two strings cross each other at most once with an
over (positive) crossing. A loop in such a tangle diagram can be replaced by a
scalar factor equal to the quantum dimension $\dim_q V$ of $V$. Also
if a string in such a tangle diagram has self-crossings,  then the
tangle diagram is proportional to that obtained by replacing this
string by one without any self-crossing.

Given a $(0, 2k)$-tangle diagram in $B_{\bf d}$,  we shall divide the
interval $[1,2k]$ into bands of length $d_1,d_2,\dots$, with
end points $1$ to $d_1$ belonging to the first band, end points
$(d_1+1)$ to $(d_1+d_2)$ to the second band,
etc.  If two strings cross each other and each has at least one end
point belonging to the same band, we may replace the crossing (over
crossing) by
$q\hspace{.1cm} \text{\begin{picture}(15, 10) \put(0,
0){\line(0, 1){10}}\put(6, 0){\line(0, 1){10}}
\end{picture}} -\  q^{1-n}\frac{q^n - q^{-n}}{\dim_q V} \text{\hspace{.2cm}\begin{picture}(10,10)
\qbezier(-5, 10)(0,3)(5,10) \qbezier(-5, 0)(0,7)(5,0)
\end{picture}}$.
Then under the map $\tau^{\otimes m}$, the resulting linear
combination of tangle diagrams yields the same element of $A_{\bf
d}$ as the original diagram. Therefore, for the purpose of
studying $A$, we may confine attention to tangle diagrams in $B$ satisfying
the further condition that strings with end points belonging to the
same band do not cross one another.

To summarise, $A_{\bf d}$ is spanned by the elements
$\tau^{\otimes m}(\phi)$, where $\phi$ is a $(0, 2k)$-tangle diagram
in $B_{\bf d}$ which satisfies the following conditions:
\begin{quote}
(a) there are no loops; \\
(b) two strings cross each other at most once with an over crossing;\\
(c) no string crosses itself;\\
(d) strings sharing a band of end points do not cross one another.
\end{quote}
As indicated above, this statement is based on the canonical
identification $B_{\bf d}\stackrel{\sim}{\longrightarrow}
\Hom_{\Uq(\fo_{2n})}(\C(q),
V^{\otimes 2k})$.

If $\zeta$ is a tangle diagram satisfying the conditions (a) to (d),
it is evident that there exist tangle diagrams $\zeta_1, \dots,
\zeta_t$ such that $\zeta=\zeta_1\cdots \zeta_t$, where $\zeta_i$
has neither crossings nor loops, e.g., is of the form
\begin{picture}(20,10)
\qbezier(0, 10)(10,-10)(20,10)
\qbezier(3, 10)(6,0)(7,10)
\qbezier(9, 10)(11,-5)(17,10)
\qbezier(11, 10)(12,0)(15,10)
\end{picture}.
Consider, by way of example, a $(0, 4)$
tangle diagram with two strings such that the end points of one
string are in bands 2 and 4, while the end points
of the other are in bands 1 and 3. In this
case, the tangle diagram is obviously the product of two $(0,
2)$-tangle diagrams.

Recall that $\Psi^{(i, j)}$ commutes with all $X_{k a}$ for $k<i<j$
or  $i<j<k$, and all $\Psi^{(i, i)}$ are central in $\cA_m$. Thus
$\tau^{\otimes m}(\zeta_i)$ can be expressed as $\tau^{\otimes
m}(\zeta_{i, 1})\tau^{\otimes m}(\zeta_{i, 2})\cdots \tau^{\otimes
m}(\zeta_{i, s})$, where the tangle diagram $\zeta_{i, j}$ has
neither crossings nor loops, and satisfies the condition that there
are two distinct bands such that every string has an end point in
each band.

Consider such a $\zeta_{i, j}$, the end points of whose strings
belong to, say, bands $k$ and $l$. If there are  only two
strings in the diagram, that is, the diagram is of the form
\begin{picture}(20,10)
\qbezier(0, 10)(8,-10)(16,10)
\qbezier(3, 10)(8,-5)(13,10)
\end{picture},
then $\tau^{\otimes m} (\zeta_{i, j})=\sum_{a, a} C_{a b} X_{k a}
\Psi^{(k, l)} X_{l b}$,  where $\Psi^{(k, l)}=
\sum_{c, d}   C_{c d} X_{k c} X_{l d}$. By the fourth relation of
\eqref{XPsi},
\[
\tau^{\otimes m} (\zeta_{i, j})= q \left(\Psi^{(k, l)}\right)^2
-\frac{(q-q^{-1})q^{n-1}}{q^{n-1}+q^{1-n}}\Psi^{(k, k)}\Psi^{(l, l)}.
\]
By induction on the number of strings in $\zeta_{i, j}$,
$\tau^{\otimes m}(\zeta_{i, j})$ can be expressed as a
linear combination of monomials in $\Psi^{(k, l)}$, $\Psi^{(k, k)}$
and $\Psi^{(l, l)}$. This completes the proof of the theorem.
\end{proof}

\section{Invariant theory of the quantum odd orthogonal groups $\Uq(\so_{2n+1})$}\label{B}

In this section we shall consider the invariant theory of
$\Uq(\mathfrak{so}_{2n+1})$. This case may be treated in almost identical
fashion to the even orthogonal case, and we shall therefore present only the
main lines of the argument, referring for details to the last section.

It will be convenient to work with a slightly modified version
of the quantum group. If $e'_i$, $f'_i$, $k'^{\pm 1}_i$ ($1\le i\le
n$) are the generators of $\Uq(\so_{2n+1})$ in the standard
Drinfeld-Jimbo presentation, where $e'_n$ and $f'_n$ are associated
with the short simple root, we write $v=q^{\frac{1}{2}}$
and let
\begin{eqnarray*}
\begin{aligned}
&e_n = \frac{e'_n}{v+v^{-1}},  \\
& e_i = e'_i,   \quad 1\le i\le n-1;  \\
&k_j=k'_j, \ \  f_j=f'_j, \quad  1\le j\le n.
\end{aligned}
\end{eqnarray*}
Then the relations among the generators remain as in the standard presentation,
except that now
\[
e_i f_j - f_j e_i = \delta_{i j} \frac{k_i - k^{-1}_i}{q-q^{-1}},
\quad \forall i.
\]
This modified quantum group is defined over $\C(q)$, and many
explicit formulae arising in its representation theory take a
correspondingly simpler form.

In this section, $\Uq$ will denote this modified algebra.

\subsection{Braided symmetric algebra of the natural $\Uq(\mathfrak{so}_{2n+1})$-module}

Denote by $V$ the natural module for $\Uq$. As above, denote by
$\{v_a \mid a\in [1, n]\cup\{0\}\cup[-n, -1]\}$ a basis
of weight vectors, where $\wt v_a=\ep_a$. Let
$E_{a b}$ be the matrix units in $\End(V)$ relative to this basis, and
denote by $\pi: \Uq\longrightarrow\End(V)$ the irreducible
representation of $\Uq$. Then relative to the above basis of $V$, $\pi$ may be
realised as follows:
\begin{eqnarray*}
\begin{aligned}
&\pi(e_i) = E_{i, i+1} - E_{-(i+1), -i},  \qquad  \pi(f_i) =E_{i+1,
i} - E_{-i, -(i+1)},  \\
&\pi(e_n) = E_{n 0} - E_{0, -n}, \quad  \pi(f_n)=E_{0 n} -
E_{-n, 0}, \\
&\pi(k_i) = 1 + (q-1)(E_{i i} + E_{-(i+1), -(i+1)}) +
(q^{-1}-1)(E_{i+1, i+1} + E_{-i, -i}),  \\
&\pi(k_n) = 1 + (q-1)E_{n n}  + (q^{-1}-1)E_{-n, -n}.
\end{aligned}
\end{eqnarray*}

The tensor product $V\otimes V$ decomposes as $L_{2\epsilon_1}\oplus
L_{\epsilon_1+\epsilon_2}\oplus L_0$;  we have the following bases
for the irreducible submodules.

\noindent (1) Basis for $L_0$ (cf. \eqref{invvxv}):
\[
\sum_{i=1}^n\left(q^{n-i}v_i\otimes v_{-i}
+ q^{i-n-1} v_{-i}\otimes v_i\right) + v_0\otimes v_0.
\]

\noindent (2) Basis for $L_{2\epsilon_1}$:
\[
\begin{aligned}
&v_{i}\otimes v_{i}, \quad v_{-i}\otimes v_{-i}, \quad 1\le i\le
n+1,
\\
&v_i \otimes v_j +q  v_j\otimes v_i, \quad v_{-j} \otimes v_{-i}
+q v_{-i}\otimes v_{-j}, \quad i<j\le n+1, \\
&v_i \otimes v_{-j} +q  v_{-j}\otimes v_i, \quad  i\ne j, \quad i, j\ne 0,\\
&(q+1)v_0\otimes v_0 - q^{-1} v_n \otimes v_{-n} - q  v_{-n}\otimes
v_n, \\
&q^{-1} v_i \otimes v_{-i} +q  v_{-i}\otimes v_i - (v_{i+1}\otimes
v_{-(i+1)} +v_{-(i+1)}\otimes v_{i+1}), \quad  i\le n-1,
\end{aligned}
\]
where $v_{n+1}$ is to be interpreted as $v_0$.

\noindent (3) Basis for $L_{\epsilon_1+\epsilon_2}$:
\[
\begin{aligned}
&v_i \otimes v_j - q^{-1} v_j\otimes v_i, \quad
v_{-j} \otimes v_{-i} - q^{-1} v_{-i}\otimes v_{-j}, \quad i<j\le n+1, \\
&v_i \otimes v_{-j} - q^{-1} v_{-j}\otimes v_i, \quad  i\ne j, \ i,j\ne 0, \\
&(q-11)v_0\otimes v_0 - v_n \otimes v_{-n} + v_{-n}\otimes
v_n, \\
&v_i \otimes v_{-i} -  v_{-i}\otimes v_i - (q v_{i+1}\otimes
v_{-i-1} -q^{-1}v_{-i-1}\otimes v_{i+1}), \quad  i\le n-1.
\end{aligned}
\]

Let $P_s$, $P_a$ and $P_0$ respectively be the idempotent
projection of $V\otimes V$
onto the irreducible submodule with highest weight $2\epsilon_1$,
$\epsilon_1+\epsilon_2$ and $0$.  Then the $R$-matrix
of $\Uq$ acting on $V\otimes V$ is given by
\begin{eqnarray}\label{checkR-odd}
\check{R}= q P_s - q^{-1} P_a + q^{-2n}P_0.
\end{eqnarray}

We also note that the quantum dimension of $V$ is given by
\[ \dim_q V = (q^{1-2n}+1)\frac{q^{2n}-q^{-1}}{q-q^{-1}}.\]

The `alternating' subspace' of the degree
$2$ homogeneous subspace $T(V)_2$ of the tensor algebra
$T(V)$ of $V$ is in this case the submodule
$L_{\epsilon_1+\epsilon_2}= P_a(T(V)_2)$. Let $\cI_q$ be the
two-sided ideal of $T(V)$ generated by $L_{\epsilon_1+\epsilon_2}$.
Then the braided symmetric algebra of $V$ is defined by
$S_q(V)=T(V)/\cI_q$ (cf. \S \ref{R2}).

\begin{lemma}
The braided symmetric algebra $S_q(V)$ is generated by $v_i, v_{-i}$
($i\in[1, n]$) and $v_{n+1}:=v_0$ subject to the following relations
\begin{eqnarray}\label{vv-relations-odd}
\begin{aligned}
&v_i  v_j - q^{-1} v_j v_i=0, \quad i<j\le n+1,\\
&v_{-j}  v_{-i} - q^{-1} v_{-i} v_{-j}=0, \quad i<j\le n+1, \\
&v_i  v_{-j} - q^{-1} v_{-j} v_i=0, \quad  i\ne j, \ i, j\ne 0\\
&(q-1)v_0 v_0 - v_n v_{-n} + v_{-n}
v_n=0, \\
&v_i  v_{-i} -  v_{-i} v_i - (q v_{i+1} v_{-i-1} -q^{-1}v_{-i-1}
v_{i+1})=0, \quad  i\le n-1.
\end{aligned}
\end{eqnarray}
\end{lemma}

The defining relations are obtained using the basis of
$L_{\epsilon_1+\epsilon_2}$ given above. From the last two
relations we obtain
\begin{eqnarray}
\begin{aligned}
v_{-i} v_i &=  v_i v_{-i} -(q-q^{-1})q^{i+1-n}\phi^{(+)}_{i+1} -
(q-1) q^{i-n}(v_0)^2,  \quad \forall i\le n,
\end{aligned}
\end{eqnarray}
where $ \phi^{(+)}_i = \sum_{k=i}^n q^{n-k}v_k v_{-k}$. From this one deduces that
\[
\sum_{i=1}^n q^{i-n-1} v_{-i} v_i = q^{1-2n}\sum_{i=1}^n q^{n-i}v_i
v_{-i} +\frac{(q-1)(q^{-2n}-1)}{q-q^{-1}}(v_0)^2.
\]
Let  $\Psi=: \sum_{i=1}^n\left(q^{n-i}v_i  v_{-i} + q^{i-n-1} v_{-i}
v_i\right) + (v_0)^2$, we have
\[
 \Psi = (1+q^{1-2n})\left(\sum_{i=1}^n q^{n-j}v_i v_{-i}
+ \frac{1-q^{-1}}{q-q^{-1}} (v_0)^2 \right).
\]

In analogy with the case of the even dimensional quantum orthogonal
group, we have the following result.
\begin{theorem}
\begin{enumerate}
\item The ordered monomials in $v_i$, $v_{-i}$ ($i\in[1, n]$) and $v_0$ form a basis of $S_q(V)$.
\item The element $\Psi$ is invariant under $\Uq(\so_{2n+1})$ and belongs to the centre of $S_q(V)$.
\item The subalgebra $S_q(V)^{\Uq(\so_{2n+1})}$ of invariants is generated by $\Psi$,
and thus is isomorphic to a polynomial algebra in one variable.
\end{enumerate}
\end{theorem}
\begin{proof}
The proof is essentially the same as that in the case of the even
dimensional quantum orthogonal group. Note that part (1) is known from \cite{BZ}.
\end{proof}
This permits the definition of the even dimensional quantum sphere,
in analogy with Lemma \ref{odd-dim-sphere}.
\begin{corollary}\label{even-dim-sphere}
Let $\langle \Psi-1\rangle$ be the two-sided ideal of $S_q(V)$
generated by $\Psi-1$, and define the quantum sphere
$\bbS_q^{2n}=S_q(V)/\langle \Psi-1\rangle$. Then $\bbS_q^{2n}$ is a
module algebra for $\Uq(\so_{2n+1})$ and $\bbS_q^{2n} =
\bigoplus_{k=0}^\infty L_{k\epsilon_1}.$
\end{corollary}

\subsection{A $\Uq(\so_{2n+1})$-module algebra and the FFT
of invariant theory}

The quantum analogue of the coordinate ring of $\oplus^m V$
is the associative algebra $\cA_m=(S_q(V)^{\otimes m},
\mu_S)$ (see \eqref{Am}), constructed by iterating the procedure of
Theorem \ref{modulealgebra}. It will be convenient to
relabel the basis elements of $V$ as follows.
Write $v_{n+1}$ for $v_0$, and
$v_{2n+2-t}$ for $v_{-t}$ for all $t\in [1, n]$.
As in \S\ref{subsec:qcoordeven}, define elements $X_{ia}\in\cA_m$
for $i\in[1, m]$ and $a\in[1, 2n+1]$.
A presentation for $\cA_m$ may be obtained as in Lemma \ref{polynomials-even}.

\begin{lemma}\label{polynomials-odd}
The $\Uq(\so_{2n+1})$-module algebra $\cA_m$ is generated by $X_{i
a}$ with $i\in[1, m]$ and $a\in[1, 2n+1]$, subject to the following
relations:

(1) for each $i$, the elements $X_{i a}$ obey the relations
\eqref{vv-relations-odd} with $v_a$ repalced by $X_{i a}$;

(2) for $i<j$ and $a, b\in[1, 2n+1]$,
\begin{eqnarray}
\begin{aligned}
X_{j a} X_{i a} &=q  X_{i a} X_{j a}, \quad a\in[1, 2n+1]\\
X_{j b} X_{i a} &= X_{i a} X_{j b}+(q-q^{-1}) X_{i b} X_{j a}, \\
X_{j a} X_{i b} &=  X_{i b} X_{j a}, \quad a<b,\  a+b\not=2n+2, \\
X_{j, n+1}X_{i, n+1} &= X_{i, n+1}X_{j, n+1} - (q-q^{-1}) \psi^{(i,
j)}_n, \\
X_{j t} X_{i, 2n+2-t} &=  q  X_{i, 2n+2-t} X_{j t} - (q-q^{-1}) q^{n-t+1} \psi_t^{(i, j)}, \\
X_{j, 2n+2-t} X_{i t}  &=  q^{-1} X_{i t} X_{j, 2n+2-t}  -  (q-q^{-1}) X_{i, 2n+2-t} X_{j t}\\
&+(q-q^{-1})q^{t-n}\left(\bar\psi_t^{(i, j)}-\Psi^{(i, j)}\right),
\quad t\in[1, n],
\end{aligned}
\end{eqnarray}
where (for all $i, j\in[1, m]$)
\[
\begin{aligned}
\psi_t^{(i, j)} &=\sum_{s=1}^t q^{s-n-1} X_{i, 2n+2-s} X_{j s},\\
\bar\psi_t^{(i, j)} &=\sum_{s=1}^t q^{n-s} X_{i s} X_{j, 2n+2-s},\\
\Psi^{(i, j)}&= \psi_n^{(i, j)}+X_{i, n+1} X_{j, n+1} +
\bar\psi_n^{(i, j)}.
\end{aligned}
\]
\end{lemma}

These relations are obtained in the same way as in the case of
$\Uq(\so_{2n})$.  The elements $\Psi^{(i, j)}$ ($i\le j$) will play
a special role in the study of invariants below; we therefore
record the following relations.
\begin{lemma}\label{PsiX-odd}
\begin{enumerate}
\item
For all $i, j\in[1, m]$, $\Psi^{(i, j)}$ are $\Uq$-invariant and
satisfy
\begin{eqnarray}
\Psi^{(j, i)} = q^{-2n} \Psi^{(i, j)} \quad \text{if \ }\ i<j.
\end{eqnarray}
\item The elements $\Psi^{(i, j)}$ satisfy the following relations:
\begin{eqnarray}\label{X-Psi}
\begin{aligned}
&X_{k a} \Psi^{(i, i)}-  \Psi^{(i, i)}X_{k a}=0, \quad \text{for all \ } i, k, \\
&X_{k a} \Psi^{(i, j)}-  \Psi^{(i, j)}X_{k a}=0, \quad k> i, j \ \ \text{or} \ k<i, j, \\
&X_{k a} \Psi^{(i, j)}-  \Psi^{(i, j)}X_{k a} \\
&= (q-q^{-1}) \left(X_{i a} \Psi^{(k, j)}- \Psi^{(i, k)}X_{j a}\right), \quad  i<k<j,\\
&\Psi^{(i, j)} X_{i a}  - q^{-1} \Psi^{(i, j)} X_{i a}
= (q-q^{-1}) \varphi^{(i, i)} X_{j a}, \quad i<j, \\
&X_{j a} \Psi^{(i, j)} - q^{-1} \Psi^{(i, j)}X_{j a} =(q-q^{-1})
\varphi^{(j, j)}X_{i a}, \quad i<j.
\end{aligned}
\end{eqnarray}
\end{enumerate}
where
\[\varphi^{(i, i)}:=\frac{q^{2n}-q^{-1}}{q-q^{-1}}\Psi^{(i, i)}=
\bar\psi_n^{(i, i)} + \frac{1-q^{-1}}{q-q^{-1}} (X_{i, n+1})^2
\]
\end{lemma}
It follows from part (2) of Lemma \ref{PsiX-odd} that the elements
$\Psi^{(i, j)}$ also satisfy the following relations:
\begin{lemma}\label{relations:generators-odd}
\begin{eqnarray}
\begin{aligned}
&\Psi^{(i, i)}\Psi^{(j, k)}-\Psi^{(j, k)}\Psi^{(i, i)}=0,
\quad \text{for all \ } \ i, j, k,\\
&\Psi^{(i, k)}\Psi^{(i, j)}-q^{-1}\Psi^{(i, j)}\Psi^{(i,
k)}=(q-q^{-1}) \varphi^{(i, i)}\Psi^{(j, k)}, \quad k\ne i, j; \ i<j, \\
&\Psi^{(j, k)}\Psi^{(i, j)}-q^{-1}\Psi^{(i, j)}\Psi^{(j, k)}=
(q-q^{-1}) \varphi^{(i, i)} \Psi^{(i, k)}, \quad k\ne i, j; \ i<j,\\
&\Psi^{(i, j)}\Psi^{(k, l)}-\Psi^{(k, l)}\Psi^{(i, j)}=0,
\quad k< i<j< l,\\
&\Psi^{(i, j)}\Psi^{(k, l)}-\Psi^{(k, l)}\Psi^{(i, j)}\\
&=(q-q^{-1})
\left(\Psi^{(i, k)}\Psi^{(j, l)} -\Psi^{(i, l)}\Psi^{(k, j)} \right),
\quad i<k<j< l.
\end{aligned}
\end{eqnarray}
\end{lemma}

We now express our results as a quantum analogue of the classical case.
Consider the quantum orthogonal algebra
$\Uq(\fo_{2n+1})$ as defined in \cite{LZ1} and discussed in
\S \ref{On}.  Extend the natural $\Uq(\so_{2n+1})$-module $V$
to a $\Uq(\mathfrak{o}_{2n+1})$-module as indicated there.
The following result is our quantum analogue of the
polynomial form of the first fundamental theorem for the invariant
theory of ${\mathfrak o}_{2n+1}$. Its proof is the same as that of
Theorem \ref{FFT-even}.
\begin{theorem}\label{FFT-odd}
The subalgebra $\cA_m^{\Uq(\fo_{2n+1})}:= \{ f\in \cA_m \mid x(f) =
\epsilon(x) f, \ \forall x\in \Uq(\fo_{2n+1})\}$ of
$\Uq(\fo_{2n+1})$-invariants in $\cA_m$ is generated by the elements
$\Psi^{(i, j)}$ ($i\le j$) and the identity. These generators satisfy
the commutation relations \eqref{relations:generators-odd}.
\end{theorem}

\section{Invariant theory for the quantum symplectic
groups $\Uq(\fsp_{2n})$}\label{CA}
We now turn to the invariant theory of the quantum symplectic group.
In this section $\Uq$ will denote $\Uq(\fsp_{2n})$.

\subsection{Braided symmetric algebra of the natural $\Uq(\fsp_{2n})$-module}

Let $V$ be the natural module for $\Uq$. As in the orthogonal cases,
$\{v_a \mid a\in [1,n]\cup[-n, -1]\}$ denotes a basis of weight vectors,
with the usual notation for weights, and
$E_{a b}$ are the matrix
units in $\End(V)$ with respect to this basis. Denote by $\pi:
\Uq(\fsp_{2n})\longrightarrow \End(V)$ the irreducible
representation of $\Uq(\fsp_{2n})$ in $V$. Then relative to the above basis, we
may realise $\pi$ as follows. In the formulae below, $i<n$.
\begin{eqnarray*}
\begin{aligned}
&\pi(e_i) = E_{i, i+1} - E_{-i-1, -i},  \quad  \pi(f_i) =E_{i,
i+1} - E_{-i, -i-1},  \\
&\pi(e_n) = E_{n, -n}, \qquad  \pi(f_n)=E_{-n, n}, \\
&\pi(k_i) = 1 + (q-1)(E_{i i} + E_{-i-1, -i-1}) +
(q^{-1}-1)(E_{i+1, i+1} + E_{-i, -i}),  \\
&\pi(k_n) = 1 + (q^2-1)E_{n n} + (q^{-2}-1)E_{-n, -n}.
 \end{aligned}
\end{eqnarray*}
For later use, we again provide explicit bases for the irreducible submodules of
$V\otimes V\simeq L_0 \oplus L_{2\epsilon_1}\oplus L_{\epsilon_1+\epsilon_2}$.
\begin{enumerate}
\item Basis of $L_0$:
\[\sum_{i=1}^n\left(q^{n-i+1}v_i\otimes v_{-i} - q^{i-n-1} v_{-i}\otimes v_i\right). \]

\item Basis of $L_{2\epsilon_1}$   ($i,
j\in[1, n]$ below):
\[
\begin{aligned}
&v_a\otimes v_a, \quad a\in [1, n]\cup[-n, -1],\\
&v_i \otimes v_j +q  v_j\otimes v_i, \quad v_{-j} \otimes v_{-i}
+q v_{-i}\otimes v_{-j}, \quad i<j, \\
&v_i \otimes v_{-j} +q  v_{-j}\otimes v_i, \quad  i\ne j, \\
& v_{i+1}\otimes v_{-i-1} +v_{-i-1}\otimes v_{i+1} - (q^{-1} v_i
\otimes v_{-i} +q  v_{-i}\otimes v_i), \quad  i<n; \\
& q^{-1} v_n \otimes v_{-n} +q  v_{-n}\otimes v_n.
\end{aligned}
\]

\item
Basis of $L_{\epsilon_1+\epsilon_2}$  ($i,
j\in[1, n]$ below):
\[
\begin{aligned}
&v_i \otimes v_j - q^{-1} v_j\otimes v_i, \quad
v_{-j} \otimes v_{-i} - q^{-1} v_{-i}\otimes v_{-j}, \quad i<j, \\
&v_i \otimes v_{-j} - q^{-1} v_{-j}\otimes v_i, \quad  i\ne j, \\
&v_i \otimes v_{-i} -  v_{-i}\otimes v_i - (q v_{i+1}\otimes
v_{-i-1} -q^{-1}v_{-i-1}\otimes v_{i+1}), \quad i<n.
\end{aligned}
\]
\end{enumerate}

Denote by $P_s$, $P_a$ and $P_0$ the idempotent projections from
$V\otimes V$ onto the irreducible submodule with highest weight
$2\epsilon_1$, $\epsilon_1+\epsilon_2$ and $0$ respectively. Then
one checks easily (cf. \cite[(6.10)]{LZ1})
that the $R$-matrix of $\Uq$
acting on $V\otimes V$ is given by
\[
\check{R} = q P_s - q^{-1}
P_a - q^{-2n-1}P_0.
\]

The `alternating subspace' of the tensor square $T_2$ of $V$
is this case therefore equal to $P_a(T_2)\oplus P_0(T_2)$.
This subspace generates a graded two-sided ideal $\cI_q$ of
$T(V)$, and as in \S\ref{bsa} the braided symmetric algebra of $V$ is
$S_q(V)=T(V)/\cI_q$; the map
$\tau: T(V)\longrightarrow S_q(V)$ is the natural surjection. As we
have seen, $S_q(V)$ has
the structure of a $\Uq$-module algebra.

The following result is easily proved
using the above explicit bases. 
\begin{lemma}\label{vv-sp}
The braided symmetric algebra $S_q(V)$ is generated by $\{v_i,
v_{-i}\mid i\in[1, n]\}$ subject to the following relations:
\begin{eqnarray}
\begin{aligned}
&v_i  v_j - q^{-1} v_j v_i=0, \quad
v_{-j}  v_{-i} - q^{-1} v_{-i} v_{-j}=0, \quad i<j, \\
&v_i  v_{-j} - q^{-1} v_{-j} v_i=0, \quad  i\ne j, \\
&v_i v_{-i} -  v_{-i} v_i - (q v_{i+1}
v_{-i-1} -q^{-1}v_{-i-1} v_{i+1})=0, \quad i<n, \\
&\sum_{i=1}^n\left(q^{n-i+1}v_iv_{-i} - q^{i-n-1} v_{-i} v_i\right)=0.
\end{aligned}
\end{eqnarray}
\end{lemma}

Write $\psi_i=\sum_{j=1}^i q^{n+1-j}v_{-i} v_i$. The last two
relations are equivalent to
\begin{eqnarray}
\begin{aligned}
&v_{-i} v_i = q^2 v_i v_{-i} + (q-q^{-1}) q^{i-n} \psi_{i-1}, &i<n,
\\
&v_{-n} v_n = v_n v_{-n} + (q-q^{-1})\psi_n.
\end{aligned}
\end{eqnarray}

It is known \cite{BZ, Zw} that $S_q(V)$ is flat, so that the degree $k$
homogeneous subspace $S_q(V)_k$ has the same dimension as that of
$S(V)_k$, which is equal to $\begin{bmatrix} 2n+k-1\\
k\end{bmatrix}$. Since $(v_1)^k$ evidently generates an irreducible
$\Uq$-submodule with highest weight $k\epsilon_1$, and this
submodule has the same dimension, it follows that
$S_q(V)_k\simeq L_{k\epsilon_1}$, and hence that
$S_q(V)=\oplus_{k=0}^\infty L_{k\epsilon_1}$ as a $\Uq$-module.
The next statement is now clear.
\begin{lemma}
The subalgebra $S_q(V)^{\Uq(\fsp_{2n})}$ of invariants is equal to
$\cK$.
\end{lemma}
\subsection{A $\Uq(\fsp_{2n})$-module algebra and the FFT of invariant theory}
As in the case of the orthogonal groups, we take as the non-commutative analogue of
the coordinate ring of $\oplus^m V$ the $\Uq$-module algebra
$\cA_m=S_q(V)^{\otimes m}$ with multiplication
obtained by iterating \eqref{multiplication}.
As in Lemma \ref{Tm}, $\tau^{\otimes m}:
\cT_m\longrightarrow \cA_m$ is a surjective homomorphism of $\Uq$-algebras.

A presentation for $\cA_m$ may be obtained by calculations analogous to
those in Lemma \ref{polynomials-even}. Define $X_{i a}$ with
$1\le i\le m$ and $a\in[-n, -1]\cup[1, n]$ in analogous fashion to
\eqref{def:X}. For all $r, s\in[1, m]$,  let
\[
\begin{aligned}
\bar\psi^{(r, s)}_i&=\sum_{i=1}^n q^{n+1-i}X_{r i}X_{s, -i}, \\
\psi^{(r, s)}_i&=\sum_{i=1}^n q^{i-n-1}X_{r, -i}X_{s i}\\
\Psi^{(s, t)}&= \sum_{i=1}^n\left(q^{n+1-i} X_{s i}X_{t, -i} -
q^{i-n-1}X_{s, -i} X_{t i}\right).
\end{aligned}
\]

\begin{lemma}\label{polynomials-sp}
The algebra $\cA_m$ is generated by the elements $X_{i a}$ ($1\le
i\le m$, $1\le a\le 2n$) subject to the following relations:

(1) for each $i$, the elements $X_{i a}$ obey the relations
\eqref{vv-sp} with $v_a$ replaced by $X_{i a}$;

(2) for $s<t\in[1, m]$, $a, b\in[1, 2n]$ and $i\in[1, n]$,
\begin{eqnarray}\label{aNb}
\begin{aligned}
&X_{t a} X_{s a} = q X_{s a} X_{t a}, \\
&X_{t a} X_{s b} =X_{s b} X_{t a},  \quad a<b\ne 2n+1-a, \\
&X_{t b} X_{s a} = X_{s a} X_{t b} +(q- q^{-1}) X_{s b} X_{t a},
\quad a<b\ne 2n+1-a;
\end{aligned}
\end{eqnarray}
\begin{eqnarray}\label{skew3}
\begin{aligned}
\Psi^{(t, s)} = -q^{-1-2n} \Psi^{(s, t)};
\end{aligned}
\end{eqnarray}
\begin{eqnarray}\label{aEb}
\begin{aligned}
X_{t i}X_{s, 2n+1-i} =&q^{-1} X_{s, 2n+1-i} X_{t i},  \\
X_{t, 2n+1-i} X_{s i} = &q X_{s i} X_{t, 2n+1-i}+ (q-q^{-1})
X_{s, 2n+1-i} X_{t i} \\
&+ (q-q^{-1})q^{i-n-1}\left(\bar\psi^{(s, t)}_{i+1} - \Psi^{(s,
t)}\right),
\end{aligned}
\end{eqnarray}
where $\psi^{(s, t)}_{n+1}=0$ by convention.
\end{lemma}

We now turn to the study of the subalgebra
$$\cA_m^{\Uq(\fsp_{2n})}:= \{ f\in \cA_m \mid x(f) = \epsilon(x) f,
\ \forall x\in \Uq(\fsp_{2n})\}$$ of $\Uq(\fsp_{2n})$ invariants in
$\cA_m$. We shall need the following result.
\begin{lemma}\label{le:sp-XPsi}
The elements $\Psi^{(i, j)}$ are $\Uq(\fsp_{2n})$-invariant,
and the following relations hold in $\cA_m$.
\begin{eqnarray}\label{sp-XPsi}
\begin{aligned}
&X_{k a} \Psi^{(i, j)} - \Psi^{(i, j)} X_{k a} =0, \quad k<i<j
\mbox{\ or\ } i<j<k, \\
& X_{k a} \Psi^{(i, j)} - \Psi^{(i, j)} X_{k a} =
(q-q^{-1}) \left(X_{i a} \Psi^{(k, j)} + \Psi^{(i, k)} X_{j a}\right), \quad i<k<j, \\
&X_{i a} \Psi^{(i, j)} - q \Psi^{(i, j)} X_{i a} = 0, \quad i<j,\\
&\Psi^{(i, j)} X_{j a} - q X_{j a}  \Psi^{(i, j)}=0, \quad i<j.
\end{aligned}
\end{eqnarray}
\end{lemma}
\begin{proof} The invariance of the elements $\Psi^{(i, j)}$ is clear.
Moreover the first and second relations are proved in exactly the same way as the
second and third relations in \eqref{XPsi}, using the skein
relations satisfied by the $R$-matrices.

To prove the other two relations, we need only consider the case
$i=1$ and $j=2$. The restrictions of $\tau^{\otimes 2}$ to
the subspaces
$1\otimes T(V)_1\oplus T(V)_1\otimes 1$ and $T(V)_1\otimes T(V)_1$
are all injective. Hence we have unique preimages
$\hat{X}_{i a}\in 1\otimes T(V)_1\oplus
T(V)_1\otimes 1$ ($i=1, 2$) and $\hat\Psi^{(1, 2)}\in T(V)_1\otimes
T(V)_1$ of $X_{i a}$ and $\Psi^{(1, 2)}$
respectively. Write $\hat\Psi^{(1,2)}=
\sum_{a,b} C_{a b} \hat X_{1 a}\hat X_{2 b}$, where $C_{a b}\in
\cK$, and let $\alpha_t,\beta_t\in\Uq$ be such that the
universal $R$-matrix $R=\sum_t
\alpha_t\otimes\beta_t$. Then
\[\hat\Psi^{(1, 2)}\hat{X}_{1 c} = \sum_{t} \sum_{a, b} C_{a b}
\hat{X}_{1 a} \beta_t(\hat X_{1 c}) \alpha_t(\hat X_{2 b}),\]
Write
$R^{-1}=\sum_t \bar\alpha_t\otimes\bar\beta_t$. Then using the
defining property $(\Delta\otimes \id)R= R_{1 3} R_{2 3}$ of $R$ and
also the fact that $\hat\Psi^{(1, 2)}$ is $\Uq$-invariant, one
shows that
\[
\hat\Psi^{(1, 2)}\hat{X}_{1 c} = \sum_{t} \sum_{a, b} C_{a b}
\bar\alpha_t(\hat X_{1 a}) \bar\beta_t(\hat X_{1 c})\hat X_{2 b}.
\]
It follows from the definition of $S_q(V)$ that $\tau \check{R}(t_2)
= q \tau(t_2) $ and $ \tau \check{R}^{-1}(t_2) = q^{-1} \tau(t_2)$
for all $t_2\in T(V)_2$. Thus
\begin{eqnarray*}
\tau^{\otimes 2} (\hat\Psi^{(1, 2)}\hat{X}_{1 c}) &= q^{-1}
\tau^{\otimes 2}( \sum_{a, b} C_{a b}\hat X_{1 c} \hat X_{1 a} \hat
X_{2 b}) &= \tau^{\otimes 2}(\hat X_{1 c} \hat\Psi^{(1, 2)}).
\end{eqnarray*}
The third relation of \eqref{sp-XPsi} follows immediately, and the
fourth relation can be proved similarly.
\end{proof}

It is straightforward to derive the following commutation relations
among the elements $\Psi^{(i, j)}$ from Lemma \ref{le:sp-XPsi}.
\begin{proposition} \label{relations:generators-C} Assume that $i<j$ and $k<l$.
\begin{enumerate}
\item If $l\ne j$, then
$\Psi^{(i, j)} \Psi^{(i, l)} = q^{-1} \Psi^{(i, l)} \Psi^{(i, j)}.$
\item \label{2} If $i<k$, then
$\Psi^{(i, j)}\Psi^{(k, j)}=q^{-1} \Psi^{(k, j)}\Psi^{(i, j)}.$
\item If $i<k$ and $j>l$, then
$\Psi^{(i, j)}\Psi^{(k, l)}= \Psi^{(k, l)}\Psi^{(i, j)}.$
\item If $i<k$ and $j<l$, then
\[\Psi^{(k, l)}\Psi^{(i, j)} - \Psi^{(i, j)}\Psi^{(k, l)} =
(q-q^{-1})\left( \Psi^{(i, l)}\Psi^{(k, j)} +\Psi^{(i, k)}\Psi^{(j,
l)}\right).\]
\end{enumerate}
\end{proposition}

The following result is the quantum analogue of the first
fundamental theorem of invariant theory for $\Uq(\fsp_{2n})$.
Its proof is a straightforward adaptation of that of Theorem \ref{o-FFT}.
\begin{theorem}\label{FFT-sp}
The subalgebra $\cA_m^{\Uq(\fsp_{2n})}$ of
$\Uq(\fsp_{2n})$-invariants in $\cA_m$ is generated by the elements
$\Psi^{(i, j)}$ ($i< j$) and the identity.
\end{theorem}

\section{Invariant theory for the quantum general linear
group}\label{GL}
In this section, we study the invariant theory of the quantum
general linear group. Our first requirement is a quantum analogue of
$S(\oplus^k V\oplus^l V^*)$, where $V$ is the  natural $\gl_n$-module
and $V^*$ is its dual. There are various ways this could be approached,
all of which turn out to be equivalent.

One way is to follow the construction of Section \ref{general} to
first construct the braided symmetric algebras $S_q(V)$ and $S_q(V^*)$  over $\Uq(\gl_n)$,
and then to form the module algebras $S_q(V)^{\otimes k}$ and
$S_q(V^*)^{\otimes l}$. From these, one
constructs the module algebra $S_q(V)^{\otimes k}\otimes
S_q(V^*)^{\otimes l}$ using Theorem \ref{modulealgebra}, and
this is a quantum analogue of the symmetric algebra over $S(\oplus^k V\oplus^l V^*)$.

Another possibility is to directly construct quantum analogues of
the symmetric algebras of $\oplus^k V$ and of $\oplus^l V^*$.  In this case,
we need to consider the Lie algebras $\gl_n\times \gl_k$ and
$\gl_n\times \gl_l$. Let $\fg=\gl_n\times \gl_k$, and let
$W=V\otimes V^{(k)}$, where $V$ and $V^{(k)}$ are respectively the
natural modules for $\Uq(\gl_n)$ and $\Uq(\gl_k)$. Then $W$ is an
irreducible $\Uq(\fg)$-module, and following \S \ref{general}
we may form the braided
symmetric algebra $S_q(W)$. It is known \cite{Zw} that $W$ is one of
the extremely rare modules such that $S_q(W)$ is flat. That is,
$S_q(W)$ is (linearly) isomorphic
to $S_q(V)^{\otimes k}$. We may similarly consider the irreducible
$\Uq(\fg)$-module $W'=V^*\otimes V^{(l)}$, where now
$\fg=\gl_n\otimes\gl_l$, and construct $S_q(W')$. Then using Theorem \ref{modulealgebra}
we obtain the module algebra $S_q(W)\otimes S_q(W')$ over
$\Uq(\gl_n)$.

The above two constructions are essentially equivalent, and are
equivalent to a third construction, which is conceptually
simpler, and is the one we shall use in the present work. We
now discuss this third construction.

\subsection{Algebra of functions on the quantum general linear group}\label{sect-Uqgl}

The quantum general linear group $\Uq(\gl_N)$ is generated by
$K_b^{\pm 1}$ ($1\le b\le N$), $e_a$ and $f_a$ ($1\le a<N$). The
defining relations are essentially the same as those for
$\Uq(\fsl_N)$ except for those involving the elements $K_b^{\pm 1}$,
which are given by
\[
\begin{aligned}
&K_b e_a K_b^{-1} = (1 + (q-1)\delta_{b a} + (q^{-1}-1)\delta_{b,
a+1}) e_a, \\
&K_b f_a K_b^{-1} = (1 + (q-1)\delta_{b, a+1} +
(q^{-1}-1)\delta_{b a}) f_a,\\
&e_a f_b - f_b e_a = \delta_{a b} \frac{K_a K_{a+1}^{-1} - K_{a+1}
K_a^{-1}}{q-q^{-1}}.
\end{aligned}
\]

Let $\{v_a \mid 1\le a\le N\}$ be the standard basis of
weight vectors of the natural
module $V^{(N)}$ for $\Uq(\gl_N)$, with $\wt(v_i)=\ep_i$, and let $\pi:
\Uq(\gl_N)\longrightarrow \End(V^{(N)})$ be the usual representation of
$\Uq(\gl_N)$ relative to this basis. Define elements $t_{a b}$
($1\le a, b \le N$) of the dual vector space $\Uq(\gl_N)^*$ of
$\Uq(\gl_N)$ by $ \langle t_{a\, b}, x\rangle =\pi(x)_{a\, b}$
(the $(a,b)$ entry of the matrix $\pi(x)$),
$\forall x\in \Uq(\gl_N)$, and call these the cooordinate functions of the
representation $\pi$.

Now the dual space of $\Uq(\gl_N)$ has a natural algebraic structure
with multiplication defined as follows. For any $t, t' \in \Uq(\gl_N)^*$,
$ \langle t\,t',\,x\rangle =\sum_{(x)} \langle t\otimes t', \,
x_{(1)} \otimes x_{(2)}\rangle,$ for all $x\in \Uq(\gl_N). $
Following an idea of \cite{Z98}, we consider the subalgebra $\cM_N$ of
$\Uq(\gl_N)^*$ generated by the elements $t_{a b}$ ($a,b\in[1, N]$).
Then $\cM_N$ has a presentation with generators $t_{a b}$
($a,b\in[1, N]$) and relations
\begin{eqnarray}\label{RTT}
\sum_{a', b'} R_{a a', b b'} t_{a' c} t_{b' d} = \sum_{a', b'} t_{b
b'} t_{a a'} R_{a' c, b' d},
\end{eqnarray}
where the $R_{a a', b b'}$ are the entries of the $R$-matrix,
which in our case has the form
\be\label{glR}
R=1\otimes 1 + (q-1)\sum_{a=1}^N E_{a a}\otimes E_{a a} +
(q-q^{-1})\sum_{a< b}E_{a b}\otimes E_{b a}.
\ee

This algebra has been studied extensively in the literature
(cf., e.g. \cite{Z03,BG,GPS}). We
shall recall a few relevant facts. The algebra
$\cM_N$ is $\Z_+$-graded with $t_{a b}$ of degree $1$. It has
the structure of a bi-algebra with co-multiplication given by
\[
\Delta(t_{a b})=\sum\limits_{c=1}^Nt_{a c}\otimes t_{c b}.
\]

There exist the following left actions $\cR, \cL: \Uq(\gl_N)\otimes
\cM_N \longrightarrow \cM_N$ of $\Uq(\gl_N)$ on $\cM_N$ respectively
defined by
\begin{eqnarray}\label{LR-actions}
\begin{aligned}
\cR: x\otimes f \mapsto \sum_{(f)}  f_{(1)}\langle f_{(2)},
x\rangle, &\quad & \cL: x\otimes f \mapsto \sum_{(f)}\langle
f_{(1)}, S(x)\rangle f_{(2)},
\end{aligned}
\end{eqnarray}
where Sweedler's notation $\Delta(f)=\sum_{(f)}f_{(1)}\otimes
f_{(2)}$ is used for the co-multiplication of $f\in \cM_N$. We
refer to $\cL$ and $\cR$ as left and right translation
respectively. It is obvious but important to observe that the two
actions commute. Both actions preserve the algebraic structure of
$\cM_N$, that is, $\cM_N$ is a $\Uq(\gl_N)$-module algebra with
respect to both the $\cL$ and $\cR$ actions. However, note that
while the $\cR(\Uq(\gl_N))$-module algebra structure
is defined with respect to the co-multiplication $\Delta$,
the $\cL(\Uq(\gl_N))$-module algebra structure  is defined
with respect to the opposite
co-multiplication $\Delta'$.

Denote by $L_\lambda^{(N)}$ the irreducible $\Uq(\gl_N)$-module with
highest weight $\lambda$. The tensor powers of $V^{(N)}$ decompose into
a direct sums of simple modules with highest weight corresponding to
partitions of length less than or equal to $N$, and every
$L_\lambda^{(N)}$ with $\lambda$ such a partition appears in
some tensor power of $V^{(N)}$ as a submodule. It follows that
we have the following multiplicity free decomposition of
$\cM_N$, which may be thought as part of a quantum Peter-Weyl
theorem.
\begin{theorem}\cite{Z03}\label{Peter-Weyl}
As a $\cL\left(\Uq(\gl_N)\right)\otimes
\cR\left(\Uq(\gl_N)\right)$-module,
\[
\cM_N\cong\bigoplus_{\lambda}
\left(L_{\lambda}^{(N)}\right)^*\otimes L_\lambda^{(N)},
\]
where the sum is over all partitions of lengths $\le N$.
\end{theorem}

Given any order on the pairs $(a, b)$ with $a, b\in [1, N]$, one
observes by inspection of the relations \eqref{RTT} that the ordered
monomials in $t_{a b}$ span $\cM_N$. Since $\dim L_\lambda^{(N)}$ is
just the dimension of the irreducible $\gl_N$-module with
highest weight $\lambda$, $\sum_{|\lambda|=j}\left(\dim
L_\lambda^{(N)}\right)^2 =
\begin{bmatrix}N^2+j-1\\ j\end{bmatrix}$, where $|\lambda|$ is the
sum of the parts of $\lambda$. The right side is equal to the number of
ordered monomials of degree $j$ in the $t_{a b}$. It follows that
the ordered monomials are linearly independent, and hence form a
basis of $\cM_N$, as is well known.

Let $\tau$ be the $\cK$-linear algebra anti-automorphism of
$\Uq(\gl_N)$ defined by
\begin{equation}\label{anti-auto}
\tau(e_a)=f_a, \quad \tau(e_a)=f_a, \quad \tau(K_b)=K_b, \quad
\forall a, b.
\end{equation}
Then the map $\tilde\cL: \Uq(\gl_N)\otimes \cM_N \longrightarrow
\cM_N$ defined by $x\otimes f\mapsto \sum_{(f)}\langle f_{(1)},
\tau(x)\rangle f_{(2)}$ gives rise to another action of $\Uq(\gl_N)$ on
$\cM_N$, and $\cM_N$ is a module algebra over $\Uq(\gl_N)$
under this action.

For a fixed positive integer $t\le N$,  the elements $K_j^{\pm 1}$,
$e_i$ and $f_i$ with $1\le j\le t$ and $1\le i<t$ generate a Hopf
subalgebra $\Uq(\gl_t)$ of $\Uq(\gl_N)$. Let $\Upsilon_t$ be the
subalgebra of $\Uq(\gl_N)$ generated by $K_a$ with $t+1 \le a \le
N$. Then $\Upsilon_t$ commutes with the subalgebra $\Uq(\gl_t)$. For
each pair of positive integers $t, s\le N$, we define the following
subalgebra of $\cM_N$:
\begin{eqnarray}\label{Mst}
\cM_{s, t} =
\left(\cM_N\right)^{\cL(\Upsilon_s)\otimes\cR(\Upsilon_t)}.
\end{eqnarray}
\begin{lemma}\label{Mst-structure}
The subalgebra $\cM_{s, t}$ of $\cM_N$ is generated by the elements $t_{i a}$
with $1\le i\le s$ and $1\le a\le t$, which satisfy the following
relations:
\begin{eqnarray} \label{ordered-RTT}
\begin{aligned}
& t_{j b} t_{i a} = \sum_{a', b'=1}^t t_{i b'} t_{j a'} R_{a' b, b'
a}, \quad i<j\le s, \\
& q t_{i a} t_{i b} = \sum_{a', b'=1}^t t_{i b'} t_{i a'} R_{a' a, b'
b},
\end{aligned}
\end{eqnarray}
Furthermore, $\cM_{s, t}$ has a basis consisting of ordered
monomials in these elements.
\end{lemma}
\begin{proof}
Let $\prod_s=\prod_{a=s+1}^N K_a$ and $\prod_t=\prod_{a=t+1}^N K_a$.
Then $f$ belongs to $\cM_{s, t}$ if and only if
$\cL_{\prod_s}(f)=\cR_{\prod_t}(f)=f$. Since
\[ \cL_{\prod_S}(t_{a b}) = \left\{ \begin{array}{l l}
t_{a b}, & a\le s,\\
q^{-1} t_{a b}, & a>s,
\end{array} \right.
\quad \cR_{\prod_t}(t_{a b}) = \left\{ \begin{array}{l l}
t_{a b}, & b\le t,\\
q t_{a b}, & b>t,
\end{array} \right.
\]
we have $t_{a_1 b_1}t_{a_2 b_2}\cdots t_{a_k b_k}\in\cM_{s, t}$ if
and only if $1\le a_i\le s$ and $1\le b\le t$ for all $i$. Such
elements obviously form a basis of $\cM_{s, t}$. This shows that the
subalgebra $\cM_{s, t}$ is generated by the elements $t_{i j}$ with
$1\le i\le s$ and $1\le j\le t$.

The relations \eqref{ordered-RTT} follow from \eqref{RTT} and the
explicit form of the $R$-matrix \eqref{glR}.
\end{proof}

\begin{remark} Let $V^{(t)}$ be the natural module for $\Uq(\gl_t)$
and let $S_q(V^{(t)})$ be the braided symmetric algebra over
$V^{(t)}$.  Consider the algebra $S_q(V^{(t)})^{\otimes s}$
discussed above.  The relations
\eqref{ordered-RTT} show that this algebra is isomorphic to $\cM_{s,
t}$, as was pointed out in \cite{B}.
\end{remark}

The next result is the quantum Howe duality (cf. \cite{Z03}) applied
to the subalgebra $\cM_{s, t}$ of $\cM_N$.
\begin{theorem}\label{truncate}
\begin{enumerate}
\item \label{truncate1} The subalgebra $\cM_{s, t}$ of $\cM_N$ is
a $\cL(\Uq(\gl_s))\otimes \cR(\Uq(\gl_t))$-module algebra, and
\begin{eqnarray}\label{def:Mst}
\cM_{s, t}\cong\bigoplus_{\lambda}
\left(L_{\lambda}^{(s)}\right)^*\otimes L_\lambda^{(t)},
\end{eqnarray}
where the summation is over all partitions of length $\le min(s,
t)$.

\item \label{truncate2} The subalgebra $\cM_{s, t}$ is a module
algebra over $\tilde\cL(\Uq(\gl_s))\otimes
\cR(\Uq(\gl_t))$ and
\begin{eqnarray}
\cM_{s, t}\cong\bigoplus_{\lambda} L_{\lambda}^{(s)}\otimes
L_\lambda^{(t)} \qquad \text{as  $\tilde\cL(\Uq(\gl_s))\otimes
\cR(\Uq(\gl_t))$-module},
\end{eqnarray}
where the range of the summation in $\lambda$ is the same as in part
(\ref{truncate1}).
\end{enumerate}
\end{theorem}
\begin{proof}
Consider the following assertions, which we shall prove shortly. For all partitions
$\lambda$,
\begin{eqnarray}\label{Trfunctor}
\begin{aligned}
&\left(L_\lambda^{(N)}\right)^{\prod_t} = \left\{
\begin{array}{l l}
L_\lambda^{(t)}, & \text{if the length of $\lambda$ is $\le t$}, \\
0, & \text{otherwise},
\end{array}\right. \\
&\left(\left(L_\lambda^{(N)}\right)^*\right)^{\prod_s} = \left\{
\begin{array}{l l}
\left(L_\lambda^{(s)}\right)^*, & \text{if the length of $\lambda$ is $\le s$}, \\
0, & \text{otherwise}.
\end{array}\right.
\end{aligned}
\end{eqnarray}
Granted \eqref{Trfunctor}, part (\ref{truncate1}) of the theorem
follows immediately.  Part (\ref{truncate2}) then follows from the
second statement in part (\ref{truncate1}).

We turn to the proof of equation \eqref{Trfunctor}, for which we provide
details, since similar arguments will be used later. Consider the
subalgebras $\Uq(\fl_t)$ and $\Uq(\fp_t)$ of $\Uq(\gl_N)$, defined as follows.

$\bullet$ $\Uq(\fl_t)$ is generated by $K_a^{\pm 1}$ (for all $a$)
and $e_b$, $f_b$ ($b\ne t$);

$\bullet$ $\Uq(\fp_t)$ is generated by the elements of $\Uq(\fl_t)$
and $e_t$.\\
The irreducible $\Uq(\gl_N)$-module $L_\lambda^{(N)}$ with highest
weight $\lambda=(\lambda_1, \lambda_2, \dots, \lambda_N)$ may be
constructed as follows. Let $L_\lambda^{(N), 0}$ be the
irreducible $\Uq(\fp_t)$-module with highest weight $\lambda$, and
construct the generalised Verma module
$\cV_\lambda=\Uq(\gl_N)\otimes_{\Uq(\fp_t)}L_\lambda^{(N), 0}$. Then
the quotient of $\cV_\lambda$ by its unique maximal submodule is
isomorphic to $L_\lambda^{(N)}$.

If $\mu$ is a weight of the generalised Verma module but not of
$L_\lambda^{(N), 0}$, then it follows from the definition of
$\cV_\lambda$ that $\mu=\lambda-\sum_{i\le t<\beta} k_{i \beta}
(\epsilon_i-\epsilon_\beta)$ for some nonnegative integers $k_{i
\beta}$ which are not all nonzero. Assume that $\lambda\in\Z_+^N$.
Then $\left(\cV_\lambda\right)^{\prod_t}\ne 0$ if and only if
$\lambda_{t+1}= \lambda_{t+2} = \dots =\lambda_N=0$ and if this holds,
$\left(\cV_\lambda\right)^{\prod_t}=L_\lambda^{(N), 0}$. It
follows that
$\left(L_\lambda^{(N)}\right)^{\prod_t}=L_\lambda^{(N), 0}$. Note
that $L_\lambda^{(N), 0}$ is an irreducible $\Uq(\fl_t)$-module and
$e_t$ acts by $0$. It restricts to the irreducible module over
$\Uq(\gl_t)$ with highest weight $(\lambda_1, \dots,
\lambda_t)$, of which it is the inflation to $\Uq(\fp_t)$.

Now consider the subalgebras $\Uq(\fl_s)$ and
$\Uq(\bar\fp_s)$ of $\Uq(\gl_N)$ defined by

$\bullet$ $\Uq(\fl_s)$ is generated by $K_a^{\pm 1}$ (for all $a$)
and $e_b$, $f_b$ ($b\ne s$);

$\bullet$ $\Uq(\bar\fp_s)$ is generated by the elements of
$\Uq(\fl_s)$ and $f_s$.\\
Let $\Gamma_{-\lambda}^{(N), 0}$ be the irreducible
$\Uq(\fp_t)$-module with lowest weight $-\lambda$, and
$\overline{\cV}_{-\lambda}=\Uq(\gl_N)\otimes_{\Uq(\bar\fp_s)}
\Gamma_{-\lambda}^{(N),0}$ the corresponding generalised Verma module.
Then the quotient of $\overline{\cV}_{-\lambda}$ by its unique
maximal submodule is the irreducible $\Uq(\gl_N)$-module
$\Gamma_{-\lambda}^{(N)}$ with lowest weight $-\lambda$; that is,
$\Gamma_{-\lambda}^{(N)}= \left(L_\lambda^{(N)}\right)^*$.

Now if $\nu$ is a weight of $\overline{\cV}_{-\lambda}$ but not of
$\Gamma_{-\lambda}^{(N), 0}$, then $\nu= -\lambda+\sum_{i\le
s<\beta} k_{i \beta} (\epsilon_i-\epsilon_\beta)$ for some
nonnegative integers $k_{i \beta}$ which are not all nonzero. Assume
$\lambda\in\Z_+^N$ and is dominant. Then arguing as above
one sees that
$\left(\Gamma_{-\lambda}^{(N)}\right)^{\cL_{\prod_s}} \cong
\left(\bar\cV_{-\lambda}\right)^{\cL_{\prod_s}}
=\Gamma_{-\lambda}^{(N), 0}$ if $\lambda=(\lambda_1, \dots,
\lambda_s, 0, \dots, 0)$ and
$\left(\Gamma_{-\lambda}^{(N)}\right)^{\cL_{\prod_s}} =0$ otherwise.
In the former case, $\Gamma_{-\lambda}^{(N), 0}$ restricts to the
irreducible $\Uq(\gl_s)$-module with lowest weight $-(\lambda_1,
\dots, \lambda_s)$, that is, to $\left(L_{(\lambda_1, \dots,
\lambda_s)}^{(s)}\right)^*$.

This completes the proof of Theorem \ref{truncate}.
\end{proof}

Let $\bar{t}_{a b}\in \Uq(\gl_N)^*$ ($a, b\in[1, N]$) be defined by

\[
\langle \bar{t}_{a b}, x\rangle = \langle t_{b a}, S(x)\rangle,
\quad \forall x\in\Uq(\gl_N).
\]
These are the coefficient functions of the dual $\bar\pi$ of the natural
representation $\pi$ of $\Uq(\gl_N)$ in $V^*$. The subalgebra
$\overline{\cM}_N$ of $\Uq(\gl_N)^*$ generated by these elements is
also a bi-algebra with comultiplication defined by
\[ \Delta(\bar{t}_{a b}) = \sum_{c=1}^N \bar{t}_{a c}\otimes
\bar{t}_{c b}.
\]
There exist left actions \[\cL, \cR: \Uq(\gl_N)\otimes
\overline{\cM}_N \longrightarrow \overline{\cM}_N\] defined in the
same way as \eqref{LR-actions}, and with respect to each of these
actions, $\overline{\cM}_N$ is a $\Uq(\gl_N)$-module algebra.

The elements $t_{a b}$ and $\bar{t}_{a b}$ ($a, b\in[1, N]$)
together generate a subalgebra $\cK[GL_q(N)]$ of the $\Uq(\gl_N)^*$,
which is clearly a bi-algebra.  The following relations
hold in $\cK[GL_q(N)]$.
\[ \sum_{c=1}^N t_{a c} \bar{t}_{b c} = \sum_{c=1}^N \bar{t}_{a c}t_{b
c}=\delta_{a b}.
\]
It follows that $\cK[GL_q(N)]$ has the
structure of a Hopf algebra with antipode given by
\[
S(t_{a b})=\bar{t}_{b a}, \quad  S(\bar{t}_{b a})=q^{2(a-b)}t_{a b}.
\]
This Hopf algebra could be thought as the algebra of regular
functions on a quantum analogue of the general linear group.

For each pair of positive integers $t, s\le N$, define the
following subalgebra of $\overline{\cM}_N$ (cf. \eqref{Mst}):
\begin{eqnarray}\label{Mstbar}
\overline{\cM}_{s, t} =
\left(\overline{\cM}_N\right)^{\cL(\Upsilon_s)\otimes\cR(\Upsilon_t)}.
\end{eqnarray}
Note that $\overline{\cM}_{s, t} =\overline{\cM}_N$ if $s=t=N$.
Using the fact that $\bar{t}_{a b}=S(t_{b a})$ we deduce from
\eqref{RTT} the following result.
\begin{lemma}
The subalgebra $\overline{\cM}_{s, t}$ is generated by the elements
$\bar{t}_{i \alpha}$ with $1\le i\le s$ and $1\le a\le t$, which
satisfy the following relations:
\begin{eqnarray} \label{ordered-RTT-dual}
\begin{aligned}
&\bar{t}_{j b}  \bar{t}_{i a} = \sum_{a', b'=1}^t \bar{t}_{i a'}
\bar{t}_{j b'} R^{-1}_{a a', b
b'}, \quad i<j\le s, \\
& q^{-1}\bar{t}_{i b}  \bar{t}_{i a} = \sum_{a', b'=1}^t \bar{t}_{i
a'} \bar{t}_{i b'} R^{-1}_{a a', b b'}.
\end{aligned}
\end{eqnarray}
Furthermore, $\overline{\cM_{s, t}}$ has a basis consisting of
ordered monomials in these elements.
\end{lemma}

Denote by $\overline{V}^{(t)}$ the dual of the natural module for
$\Uq(\gl_t)$, and let $S_q(\overline{V}^{(t)})$ be the braided
symmetric algebra over $\overline{V}^{(t)}$. Define algebraic
structures on tensor powers of
$\overline{S}_q:=S_q(\overline{V}^{(t)})$ recursively by
\begin{eqnarray}
\mu_{k+1}:=\mu_k\otimes PR^{-T}\otimes \mu_1:
\overline{S}_q^{\otimes k}\otimes \overline{S}_q\longrightarrow
\overline{S}_q^{\otimes k}\otimes \overline{S}_q,
\end{eqnarray}
where $\mu_1$ denotes the multiplication of $\overline{S}_q$, and
$R^{-T}=\sum_s \bar\beta_s\otimes\bar\alpha_s$ if we write the
inverse of the R-matrix as $R^{-1}=\sum_s
\bar\alpha_s\otimes\bar\beta_s$. It is easy to show that $R^{-T}$
also satisfies all the relations \eqref{R1} and \eqref{R2}. Thus
Theorem \ref{modulealgebra} and Lemma \ref{hyper-associativity} are
valid when $R$ is replaced by $R^{-T}$. It follows that
$({S_q(\overline{V}^{(t)})}^{\otimes k}, \mu_k)$ has the
structure of a module algebra for $\Uq(\gl_t)$ for any $k$.

\begin{corollary}
(1) For any fixed $i$, the subalgebra of $\overline{\cM}_{s, t}$ generated by the
elements $\bar{t}_{i a}$ ($1\le a\le t$) is
isomorphic to $S_q(\overline{V}^{(t)})$.

(2) As a $\Uq(\gl_t)$-module algebra, $\overline{\cM}_{s, t}$
is isomorphic to ${S_q(\overline{V}^{(t)})}^{\otimes s}$ with
multiplication $\mu_s$.
\end{corollary}
\begin{proof}
The second relation in \eqref{ordered-RTT-dual} implies (1),
while the first relation implies (2).
\end{proof}

The following result is now clear.
\begin{theorem}\label{truncate:barM}
The subalgebra $\overline{\cM}_{s, t}$ is generated by the elements
$\bar{t}_{i j}$ with $1\le i\le s$ and $1\le j\le t$, and has a
basis consisting of ordered monomials in these elements.
Furthermore, $\overline{\cM}_{s, t}$ is a
$\Uq(\gl_s)\otimes \Uq(\gl_t)$-module algebra via the action
$\cL\otimes \cR$ and has module decomposition
\begin{eqnarray}\label{def:barMst}
\overline{\cM}_{s, t}\cong\bigoplus_{\lambda} L_\lambda^{(s)}\otimes
\left(L_{\lambda}^{(t)}\right)^*,
\end{eqnarray}
where the summation is over all partitions of length $\le min(s,
t)$.
\end{theorem}

\subsection{Invariant theory for the quantum general linear group}

Fix positive integers $n, k, l\leq N$.
We shall work with the following $\Uq(\gl_n)$-module
algebra.

\begin{definition}
Recall that we have $\cR(\Uq(\gl_n))$-module algebras
$\cM_{k, n}$ and $\overline{\cM}_{l, n}$,
respectively defined in \eqref{Mst} and
\eqref{Mstbar}. If we define multiplication
 in $\cA_{k, l}: =\cM_{k,
n}\otimes \overline{\cM}_{l, n}$ by \eqref{multiplication}
(cf. Theorem \ref{modulealgebra})
$\cA_{k, l}$ becomes a $\Uq(\gl_n)$-module algebra, with
action defined for any $f\otimes g\in \cA_{k, l}$ and
$x\in\Uq(\gl_n)$ by
\[
x(f\otimes g) = \sum_{(x)} \cR_{x_{(1)}}(f)\otimes \cR_{x_{(2)}}(g).
\]
\end{definition}
We investigate the structure of $\cA_{k,l}$. Let $X_{i a}:=t_{i a}\otimes 1$ and $Y_{\beta
b}:=1\otimes \bar{t}_{\beta b}$ with $i\in[1, k]$, $a, b\in[1, n]$
and $\beta\in[1, l]$. Then we have
\begin{eqnarray}
Y_{b \beta} X_{i a} = \sum_{a', b'=1}^n (R^{-T})_{a' b, b' a} X_{i
b'} Y_{a' \beta},
\end{eqnarray}
where
$
R^{-T} =1\otimes 1 + (q^{-1}-1)\sum_{a=1}^n E_{a a}\otimes
E_{a a} - (q-q^{-1})\sum_{1\le a< b\le n}E_{b a}\otimes E_{a b}.
$
These relations can be made more explicit, as follows.
\begin{eqnarray}
\begin{aligned}
&Y_{\beta a} X_{i a} = q X_{i a} Y_{\beta a} - (q-q^{-1})
\sum_{c=a}^n X_{i c} Y_{\beta c}, \\
& Y_{\beta b} X_{i a} =  X_{i a} Y_{\beta b}, \quad  \text{if\ }
a\ne b.
\end{aligned}
\end{eqnarray}

Our purpose is to study the subalgebra of
$\cA_{k, l}$ comprising its $\Uq(\gl_n)$-invariant elements:
$$\left(\cA_{k, l}\right)^{\Uq(\gl_n)}=\left\{ f\in \cA_{k, l} \mid
x(f) = \epsilon(x) f, \forall x\in\Uq(\gl_n) \right\}.$$
\begin{lemma}\label{relations:generators-gl}
Let $\Psi_{i \beta}:=\sum_{a=1}^n X_{i a} Y_{\beta a}$ for $1\le
i\le k$ and $1\le \beta\le l.$
\begin{enumerate}
\item The elements $\Psi_{i \beta} $ ($1\le i\le k$, $1\le \beta\le
l$) belong to $\left(\cA_{k, l}\right)^{\Uq(\gl_n)}$, the invariant
subalgebra of $\cA_{k, l}$.

\item The following relations hold in $\cA_{k, l}$ among the elements
$\Psi_{i \beta}$, $X_{i a}$ and $Y_{\alpha b}$.
\begin{eqnarray}
\begin{aligned}
&\Psi_{j \beta}X_{i a} =X_{i a} \Psi_{j \beta}, &\quad i<j, \\
&X_{j a} \Psi_{i \beta} -  \Psi_{i \beta}X_{j a}= (q-q^{-1})
X_{i a} \Psi_{j \beta}, &\quad i<j,\\
&\Psi_{i \beta}X_{i a} =q^{-1} X_{i a} \Psi_{i \beta},\\
&\Psi_{j \beta}Y_{\alpha b} =Y_{\alpha b}\Psi_{j \beta}, &\quad \alpha<\beta, \\
&\Psi_{j \alpha} Y_{\beta b} -   Y_{\beta b} \Psi_{j \alpha}=
(q-q^{-1})
Y_{\alpha b} \Psi_{j \beta}, &\quad \alpha<\beta,\\
&\Psi_{i \beta}Y_{\beta b} =q Y_{\beta b} \Psi_{i \beta}.
\end{aligned}
\end{eqnarray}

\item The elements $\Psi_{i \beta}$ satisfy the following relations.
\begin{eqnarray}
\begin{aligned}
&\Psi_{j \beta} \Psi_{i \alpha} = \Psi_{i \alpha} \Psi_{j \beta}, &&
i<j, \ \alpha< \beta, \\
&\Psi_{j \alpha} \Psi_{i \beta} - \Psi_{i \beta} \Psi_{j \alpha} =
(q-q^{-1})\Psi_{i \alpha} \Psi_{j \beta}, &&
i<j, \ \alpha< \beta, \\
&\Psi_{i \beta} \Psi_{i \alpha} = q^{-1}  \Psi_{i \alpha} \Psi_{i
\beta}, && \alpha< \beta, \\
&\Psi_{j \beta} \Psi_{i \beta} = q  \Psi_{i \beta} \Psi_{j \beta},
&& i<j.
\end{aligned}
\end{eqnarray}
\end{enumerate}
\end{lemma}
\begin{proof} Part (1) is a special case of \eqref{casimir}, but
we provide an explicit computation. Let $\pi$ be the natural representation of
$\Uq(\gl_n)$ and $\bar\pi$ its dual. For all
$x\in\Uq(\gl_n)$, we have
\[ x\Psi_{i \alpha} =\sum_{a, a' b'=1}^n \sum_{(x)} X_{i a'} Y_{\alpha
b'} \pi(x_{(1)})_{a' a} \bar{\pi}(x_{(2)})_{b' a} =\epsilon(x)
\Psi_{i \alpha}.
\]
This proves part (1).

All the relations in part (2) are proved in a similar way by
routine calculations. We illustrate the proof by considering the
second last relation. If $\alpha<\beta$, then
$
\Psi_{j \alpha} Y_{\beta b} = \sum_{a, a', b'=1}^n X_{j a} R_{a a',
b b'} Y_{\beta b'} Y_{\alpha a'},
$
where $R_{a a', b b'}$ are the entries of the $R$-matrix acting
on the tensor square of the
{\em natural representation} of $\Uq(\gl_n)$. Let $P$ denote the
permutation map on this space.
Then the above relation may be written
\[
\Psi_{j \alpha} Y_{\beta b} = \sum_{a, a', b'=1}^n X_{j a}
(\check{R}P)_{b b', a a'} Y_{\beta b'} Y_{\alpha a'}.
\]
We also have, similarly,
\[
Y_{\beta b} \Psi_{j \alpha}  = \sum_{a, a', b'=1}^n X_{j a}
(R^{-1})_{b b', a a'} Y_{\beta b'} Y_{\alpha a'}.
\]
Using the skein relation $\check{R} - \check{R}^{-1}= q-q^{-1}$,
where $\check R=PR$, we obtain
\[
\begin{aligned}
\Psi_{j \alpha} Y_{\beta b} - Y_{\beta b} \Psi_{j \alpha} &=
(q-q^{-1}) \sum_{a, a', b'=1}^n X_{j a} P_{b b', a a'} Y_{\beta b'}
Y_{\alpha a'}\\
&= (q-q^{-1}) \Psi_{j \beta} Y_{\alpha b},
\end{aligned}
\]
which proves the second last relation of part (2) in view of the fourth relation.

Part (3) is a straightforward consequence of part (2).
\end{proof}

We shall prove the following result, which is the quantum analogue of the
first fundamental theorem of invariant theory for the quantum
general linear group.
\begin{theorem} \label{gl:main}\label{FFT-gl}  The invariant subalgebra $\left(\cA_{k,
l}\right)^{\Uq(\gl_n)}$ is generated by the elements $ \Psi_{i
\beta} $ with $ 1\le i\le k, \ 1\le \beta\le l$.
\end{theorem}
The proof will be given in the next section.

\subsection{Proof of the FFT for the quantum general linear group}

This section is devoted to the proof of Theorem \ref{FFT-gl}. The techniques we use are
similar to those in \cite{GL1}. However, the object corresponding to
$\cA_{k, l}$ in \cite{GL1} is not a $\Uq(\gl_n)$-module algebra, so that our
concept of being `generated by' a subset is quite different from the set up of \cite{GL1}.

\subsubsection{The case $n\ge max(k, l)$}\label{n-ge-kl}
We consider the algebra $\cM_{s t}$ defined by \eqref{Mst} with $N=n$,
$s=k$ and $t=l$. Then $\Delta(\cM_{k,
l})\subset \cM_{k, n}\otimes \cM_{n, l}$ under the co-multiplication
on $\cM_n$, as can be seen from Theorem
\ref{truncate}(\ref{truncate1}) and  the formula
$
\Delta(t_{i \beta}) = \sum_{a=1}^n t_{i a}\otimes t_{a \beta},
$
for all $i\in[1, k]$ and $\beta\in[1, l]$.
It follows that $(\id\otimes S)\Delta(\cM_{k, l}) \subset \cA_{k,
l}$.

\begin{lemma}\label{lem:injcet}
The $\cK$-linear map
\[
\Delta_{k, l}: \cM_{k, l}
\longrightarrow \cA_{k, l}
\]
given by $\Delta_{k, l}(f) := (\id\otimes S)\Delta(f)$,
is injective and satisfies $\Delta_{k, l}(\cM_{k, l})$ $\subseteq$
$\left(\cA_{k, l}\right)^{\Uq(\gl_n)}$.
\end{lemma}
\begin{proof}
To prove injectivity of $\Delta_{k, l}$, we note that for all $f\in
\ker\Delta_{k, l}$,
\[
(\epsilon\otimes\id)\Delta_{k, l}(f) = \sum_{(f)} 1\otimes
\epsilon(f_{(1)}) S(f_{(2)}) =S(f) =0.
\]
Thus $\ker\Delta_{k, l}=0$ since $S$ is invertible.

Now for $x\in\Uq(\gl_n)$ and $f\in\cM_{k, l}$, we have
\[
\begin{aligned}
x\Delta_{k, l}(f) &= \sum_{(f), (x)} \cR_{x_{(1)}}(f_{(1)})\otimes
\cR_{x_{(2)}}(S(f_{(2)}))\\
&= \sum_{(f), (x)} f_{(1)} \langle f_{(2)}, x_{(1)}\rangle \otimes
\langle f_{(2)}, S(x_{(2)}\rangle S(f_{(3)})\\
&= \epsilon(x)\sum_{(f), (x)} f_{(1)} \otimes S(f_{(2)}).
\end{aligned}
\]
This proves the second claim of the lemma.
\end{proof}

It is useful to note that
$
\Delta_{k, l}(t_{i \beta}) = \Psi_{i \beta}.
$

\begin{lemma} \label{n-large} If $n\ge max(k, l)$, then $\left(\cA_{k,
l}\right)^{\Uq(\gl_n)}=\Delta_{k, l}(\cM_{k, l})$.
\end{lemma}

\begin{proof} Using parts (\ref{truncate1}) and (\ref{truncate2}) of
Theorem \ref{truncate} and also Theorem \ref{truncate:barM}, we can
show that as a module over $\tilde\cL(\Uq(\gl_k))\otimes
\cR(\Uq(\gl_n))\otimes\cL(\Uq(\gl_l))\otimes \cR(\Uq(\gl_n))$,
\[
\cA_{k, l} = \bigoplus_{\lambda, \mu} L^{(k)}_\lambda\otimes
L^{(n)}_\lambda\otimes L^{(l)}_\mu
\otimes\left(L^{(n)}_\mu\right)^*,
\]
where the summation is over partitions $\lambda,\mu$ where
$\lambda$ has length $\le k$ and $\mu$ has length
$\le l$. Since $n\ge max(k, l)$, it follows that
\[
(\cA_{k, l})^{\Uq(\gl_n)} = \bigoplus_{\lambda}
L^{(k)}_\lambda\otimes  L^{(l)}_\lambda
\]
as a module for $\tilde\cL(\Uq(\gl_k))\otimes \cL(\Uq(\gl_l))$,
where the summation is over all partitions of length $\le min(k,
l)$. Thus by
Theorem \ref{truncate}(\ref{truncate2}), $\cM_{k, l}\cong (\cA_{k,
l})^{\Uq(\gl_n)}$ as $\tilde\cL(\Uq(\gl_k))\otimes
\cR(\Uq(\gl_l))$-module. The lemma now follows from
the injectivity of $\Delta_{k, l}$, proved in Lemma \ref{lem:injcet}.
\end{proof}

In the next Lemma, we use the fact that the elements of $\cM_{k,l}$
are linear functions on $\Uq(\gl_n)$.
\begin{lemma}\label{f.g.-gl}
Assume that $n\ge max(k, l)$.
\begin{enumerate}
\item There exists the $\Z_+\times\Z_+$-graded vector space bijection
\[
\begin{aligned}
&\varpi: \cM_{k, l}\otimes \cM_{k, l}\longrightarrow \cM_{k,
l}\otimes \cM_{k, l}, \\
&f\otimes g\mapsto \sum_{(f), (g)} f_{(1)}\otimes g_{(1)} \langle
f_{(2)}\otimes g_{(2)}, R^{-1}\rangle.
\end{aligned}
\]
\item Denote by $\nu$ the multiplication in $\cM_{k, l}$.
For any $f, g\in\cM_{k, l}$, the product of $\Delta_{k,
l}(f)$ and $\Delta_{k, l}(g)$ in $\cA_{k, l}$ is given by
\[\Delta_{k, l}(f) \Delta_{k, l}(g)= \Delta_{k, l}\circ\nu\circ\varpi(f \otimes g).\]
\end{enumerate}
\end{lemma}
\begin{proof}
To prove part (1), we need only show that the image of
$\varpi$ is in $\cM_{k, l}\otimes \cM_{k, l}$, since then the
existence of the inverse map of $\varpi$ follows from the
invertibility of $R$. Note that for all $\alpha, \beta\in[1, l]$, we
have $\langle t_{a \alpha}\otimes t_{b \beta}, R^{-1}\rangle=0$
unless both $a$ and $b$ belong to $[1, l]$. This can be seen from
the explicit formula for the $R$-matrix in the tensor square of the
natural representation of $\Uq(\gl_n)$. More generally, for any
ordered monomials $f_1=t_{a_1 \alpha_1}\cdots t_{a_r \alpha_r}$ and
$f_2=t_{b_1 \beta_1}\cdots t_{b_1 \beta_1}$, one shows, using
the defining properties of $R$ under co-multiplication that
$\langle f_1\otimes f_2, R^{-1}\rangle=0$ unless all $a_i$ and $b_j$
belong to $[1, l]$. It follows from this that if
$f, g\in\cM_{k, l}$ then $\sum_{(f), (g)} f_{(1)}\otimes g_{(1)}
\langle f_{(2)}\otimes g_{(2)}, R^{-1}\rangle$ belongs to $\cM_{k,
l}\otimes \cM_{k, l}$.

Part (2) is proved as follows:
\begin{eqnarray}\label{inDeltaMkl}
\begin{aligned}
\Delta_{k, l}(f) \Delta_{k, l}(g)
&=\sum_{(f), (g)} f_{(1)}
g_{(1)}\otimes S(g_{(3)} f_{(3)}) \langle f_{(2)}\otimes g_{(2)},
R^{-1}\rangle\\
&=\sum_{(f), (g)} f_{(1)} g_{(1)}\otimes S(f_{(2)} g_{(2)}) \langle
f_{(3)}\otimes g_{(3)}, R^{-1}\rangle \\
 &= \Delta_{k, l}\left(\sum_{(f), (g)} f_{(1)}g_{(1)} \langle
f_{(2)}\otimes g_{(2)}, R^{-1}\rangle\right)\\
&=\Delta_{k, l}\circ\nu\circ\varpi(f\otimes g).
\end{aligned}
\end{eqnarray}
This completes the proof of the lemma.
\end{proof}

\begin{proof}[Proof of Theorem \ref{gl:main} when $n\ge max(k, l)$]
By Theorem \ref{truncate}(\ref{truncate1}), $\cM_{k, l}$ is
generated by the elements $t_{i \beta}$, and we have seen
that $\Delta_{k, l}(t_{i \beta})$ $ =$ $ \Psi_{i \beta}$. Now the
degree $i$ homogeneous subspace $(\cM_{k, l})_i$ ($i\ge 1$) of
$\cM_{k, l}$ is equal to $\nu\circ\varpi((\cM_{k, l})_{i-1}\otimes
(\cM_{k, l})_1)$ by Lemma \ref{f.g.-gl}(1). Using Lemma
\ref{f.g.-gl}(2), one can show by induction on the degree of
$(\cM_{k, l})_i$ that every element in $\Delta_{k, l}(\cM_{k, l})_i$
can be expressed as a linear combination of products of $ \Psi_{i
\beta}$. Now Theorem \ref{gl:main} follows from Lemma \ref{n-large}.
\end{proof}

\subsubsection{The case $n < max(k, l)$}
\begin{proof}[Proof of Theorem \ref{gl:main} when $n< max(k, l)$]
We shall treat separately the situations $min(k, l)\le n <
max(k, l)$ and $n <min(k, l)$. In both situations, we have
the following $\tilde\cL(\Uq(\gl_k))\otimes \cL(\Uq(\gl_l))$-module
decomposition of $(\cA_{k, l})^{\Uq(\gl_n)}$.

\[
(\cA_{k, l})^{\Uq(\gl_n)} = \bigoplus_{\lambda}
L^{(k)}_\lambda\otimes  L^{(l)}_\lambda,
\]
where the summation is over all partitions of length $\le m:=
min(k, l, n)$.

Let $\Upsilon_{k, m}$ be the subalgebra of $\Uq(\gl_k)$ generated by
the elements $K^{\pm 1}_s$ with $m+1\le s\le k$, and similarly
define the subalgebra $\Upsilon_{l, m}$ of $\Uq(\gl_l)$.
Using the fact that $\tilde\cL(\Uq(\gl_k))\otimes
\cL(\Uq(\gl_l))$ commutes with the $\Uq(\gl_n)$-action on $\cA_{k,
l}$ we may also define the
subalgebra
\[
\cA_m^0:= \left((\cA_{k,
l})^{\Uq(\gl_n)}\right)^{\tilde\cL(\Upsilon_{k, m})\otimes
\cL(\Upsilon_{l, m})} = \left((\cA_{k, l})^{\tilde\cL(\Upsilon_{k,
m})\otimes \cL(\Upsilon_{l, m})} \right)^{\Uq(\gl_n)},
\]
Then
$\cA_m^0$ is generated by the elements
\[ \Psi_{i j} =\sum_{c=1}^n t_{i c}\otimes
\bar{t}_{j c}, \quad i, j\in[1, m]
\]
by Section \ref{n-ge-kl}.

As a $\tilde\cL(\Uq(\gl_m))\otimes \cL(\Uq(\gl_m))$-module,
$\cA_m^0$ has the following decomposition into simples,
\[
\cA_m^0 = \bigoplus_{\lambda} L^{(m)}_\lambda\otimes
L^{(m)}_\lambda,
\]
where the summation is over all partitions of length $\le m$.
Following the proof of equation \eqref{Trfunctor}, one now shows that
the $\tilde\cL(\Uq(\gl_k))\otimes \cL(\Uq(\gl_l))$ highest weight
vectors of weight $\lambda$ in $(\cA_{k, l})^{\Uq(\gl_n)}$ are
precisely the $\tilde\cL(\Uq(\gl_m))\otimes \cL(\Uq(\gl_m))$ highest
weight vectors of the same weight in $\cA_m^0$.  Hence $\cA_m^0$
generates $(\cA_{k, l})^{\Uq(\gl_n)}$ as a
$\tilde\cL(\Uq(\gl_m))\otimes \cL(\Uq(\gl_m))$-module. That is,
\[
(\cA_{k, l})^{\Uq(\gl_n)} =\left(
\tilde\cL(\Uq(\gl_k))\otimes\cL(\Uq(\gl_l))\right)\cA_m^0.
\]

To complete the proof of Theorem \ref{gl:main} in this
case, we now need to show that every element of $\left(
\tilde\cL(\Uq(\gl_k))\otimes\cL(\Uq(\gl_l))\right)\cA_m^0$ can be
expressed as a linear combination of products of $\Psi_{i \beta}$
($i\le k$, $\beta\le l$). We shall prove this by induction on the
bi-degree in the $\Z_+\times \Z_+$-grading of $(\cA_{k,
l})^{\Uq(\gl_n)}$. It is clear for the subspaces of bi-degrees
$(0, 0)$, $(1, 0)$ and $(0, 1)$.

Now if $\zeta=\sum_r f_r\otimes g_r$ and $\zeta'=\sum_s f'_s\otimes
g'_s$ are elements of $(\cA_{k, l})^{\Uq(\gl_n)}$, it follows from
the commutativity of the left and right translations, that for
all $x\in\Uq(\gl_k)$ and $y\in\Uq(\gl_l)$,
\[
(\tilde\cL_x\otimes\cL_y)(\zeta\zeta') =\sum_{(x), (y)}
\left(\sum_{r} \tilde\cL_{x_{(1)}}(f_r)\otimes
\cL_{y_{(1)}}(g_r)\right) \left(\sum_{s}
\tilde\cL_{x_{(2)}}(f'_s)\otimes \cL_{y_{(2)}}(g'_s)\right).
\]
This formula allows us to complete the induction step. The
proof of Theorem \ref{gl:main} is now complete.
\end{proof}

\subsection{The braided exterior algebra}\label{AA}

Let $V^{(m)}$ and $V^{(n)}$ be the natural modules for
$\Uq(\mathfrak{gl}_m)$ and $\Uq(\gl_n)$ respectively. Then
$V^{(m)}\otimes V^{(n)}$ is an irreducible
$\Uq(\gl_m\times \gl_n)$-module, and
in this section, we shall study the braided exterior
algebra of this module. Throughout
this section, we denote $\gl_m\times \gl_n$ by $\fg$ and
$\Uq(\gl_m\times \gl_n)$ by $\Uq(\fg)$.

Denote the standard weight basis of the natural
$\Uq(\mathfrak{gl}_m)$-module $V^{(m)}$ by
$\{v_i^{(m)} \mid i=1, 2, \dots, m\}$. The
tensor square $V^{(m)}\otimes V^{(m)}$ has two irreducible
submodules, which we denote by $L^{(m, s)}$ and $L^{(m, a)}$
respectively, and which have the following bases.
\begin{enumerate}
\item basis for $L^{(m, s)}$:
\[ v^{(m)}_i\otimes v^{(m)}_i  (\text{for all $i$}), \quad v^{(m)}_i\otimes
v^{(m)}_j+q\otimes v^{(m)}_j \otimes v^{(m)}_i  (i<j);
\]
\item basis for $L^{(m, a)}$:
\[  v^{(m)}_i\otimes v^{(m)}_j - q^{-1} \otimes v^{(m)}_j \otimes v^{(m)}_i  (i<j).
\]
\end{enumerate}
Denote by $P^{(m, s)}$ and $P^{(m, a)}$ respectively the idempotent projections
mapping $V^{(m)}\otimes V^{(m)}$ onto the irreducible modules
$L^{(m, s)}$ and $L^{(m, a)}$ respectively. Then the $R$-matrix of
$\Uq(\gl_m)$ acting on $V^{(m)}\otimes V^{(m)}$ is given by
\[
\check{R}^{(\gl_m)} = q P^{(m, s)}- q^{-1} P^{(m, a)}.
\]

Similarly we have the natural $\Uq(\mathfrak{gl}_n)$-module $V^{(n)}$,
with standard basis $\{v^{(n)}_i \mid
i=1, 2, \dots, n\}$, and an analogous decomposition of its tensor
square $V^{(n)}\otimes V^{(n)}\cong L^{(n, s)} \oplus L^{(n, a)}$.
Let $P^{(n, s)}$ and $P^{(n, a)})$ be the corresponding idempotent projections
as above.

Write $V=V^{(m)}\otimes V^{(n)}$; this space has basis
\[
\{X_{i j}: = v_i^{(m)}\otimes v_j^{(n)} \mid i=1, 2, \dots, m, \
j=1, 2, \dots, n\}.
\]
Define the map $\sigma: V\otimes V \longrightarrow V^{(m)}\otimes
V^{(m)} \otimes V^{(n)} \otimes V^{(n)}$, $v_i^{(m)}\otimes
v_j^{(n)}\otimes v_s^{(m)}\otimes v_t^{(n)}\mapsto v_i^{(m)}\otimes
v_s^{(m)}\otimes v_j^{(n)}\otimes v_t^{(n)}$. Then the $R$-matrix of
$\Uq(\fg)=\Uq(\gl_m\times \gl_n)$ acting on $V\otimes V$ can be
expressed as
\[
\check{R}= \sigma^{-1}\left(q^2 P^{(m, s)}\otimes P^{(n, s)} +
q^{-2} P^{(m, a)} P^{(n, a)} - P^{(m, s)}\otimes P^{(n, a)} - P^{(m,
a)}\otimes P^{(n, s)} \right)\sigma.
\]
It follows that
$ (\check{R}-q^2)(\check{R}-q^{-2}) (\check{R}+1) =0. $
Let
\[
\begin{aligned}
a_2&= (\check{R}-q^2)(\check{R}-q^{-2})(V\otimes V)=\sigma(L^{(m,
a)}\otimes L^{(n, s)})
+ \sigma(L^{(m, s)}\otimes L^{(n, a)}) , \\
s_2&=(\check{R}+1)(V\otimes V) = \sigma(L^{(m, s)}\otimes L^{(n,
s)}) + \sigma(L^{(m, a)}\otimes L^{(n, a)}).
\end{aligned}
\]

Now the braided symmetric algebra $S_q(V)$ of $V$ is known to
be isomorphic to $\cM_{m, n}$, which has been
studied extensively \cite{Z03, BZ}.  Here we consider the
braided exterior algebra $\Lambda_q(V)$ defined
(in analogy with the definition of $S_q(V)$) by

\be\label{defn:exterior}
\Lambda_q(V) = T(V)/ \langle s_2\rangle,
\ee
where $\langle s_2\rangle$ is the two-sided ideal in the tensor
algebra $T(V)$ of $V$ generated by $s_2$.
\begin{proposition}\label{Lambda-def}
The associative algebra $\Lambda_q(V)$ is generated by
$$\{X_{i j} \mid i=1, 2, \dots, m, \;\;j=1, 2, \dots, n\},$$
subject to the following relations:
\begin{eqnarray}
\begin{aligned} \label{relations-aass}
&X_{i l} X_{j k} + X_{j k} X_{i l} + (q-q^{-1}) X_{j l} X_{i k}=0,  \\
&X_{i k} X_{j l} + X_{j l} X_{i k}=0,  \qquad i<j, \ k<l,
\end{aligned}\\
\begin{aligned}\label{relations-ss2}
&(X_{i k})^2=0,  \quad \forall i, k,  \\
&X_{i l}  X_{i k} + q^{-1} X_{i k} X_{i l}=0, \quad k<l, \text{ \
all
\ } i, \\
&X_{j k} X_{i k} +q^{-1} X_{i k} X_{j k}=0, \quad i<j, \text{ \ all
\ } k.
\end{aligned}
\end{eqnarray}
\end{proposition}
\begin{proof}
Note that from $\sigma(L^{(m, a)}\otimes L^{(n, a)})$, we obtain the
relation
\begin{eqnarray} \label{relations-aa}
X_{i k} X_{j l} +q^{-2} X_{j l} X_{i k} - q^{-1}
(X_{j k} X_{i l} +
 X_{i l} X_{j k}) =0, \quad i<j, \ k<l.
\end{eqnarray}
The subspace $\sigma(L^{(m, s)}\otimes L^{(n, s)})$ of the space
of relations includes the relations
\eqref{relations-ss2}, as well as
\begin{eqnarray}\label{relations-ss1}
X_{i k} X_{j l} +q^2 X_{j l} X_{i k} + q (X_{j k} X_{i l} +
 X_{i l} X_{j k}) =0, \quad i<j, \ k<l.
\end{eqnarray}
The relations \eqref{relations-aa} and \eqref{relations-ss1} are
equivalent to \eqref{relations-aass}.
\end{proof}

\begin{proposition}\label{Lambda-gl} The set $\{X(\ep)\}$
constitutes a basis of the braided exterior algebra $\Lambda_q(V)$, where
$\ep$ runs over sequences of the form $\ep=(\ep_1,\ep_2,\dots,\ep_{mn})$,
with $\ep_i=0 \text{ or }\pm 1$ for $i=1,\dots, mn$ and
$X(\ep)=\\(X_{1 1})^{\epsilon_{1 1}} (X_{1 2})^{\epsilon_{1 2}} \dots (X_{1
n})^{\epsilon_{1 n}} (X_{2 1})^{\epsilon_{21}} (X_{2
2})^{\epsilon_{2 2}} \dots (X_{2 n})^{\epsilon_{2 n}}
 \dots (X_{m 1})^{\epsilon_{m 1}}
\dots (X_{m n})^{\epsilon_{m n}}.$

\end{proposition}
\begin{proof}
The fact that these elements span $\Lambda_q(V)$ is a
consequence of Proposition \ref{Lambda-def}. To prove that they are
linearly independent, we consider the Grassmann algebra
$\cK[\theta]$ over $\cK$. This is generated by $\{\theta_{i j}$ ($1\le i\le m;
1\le j\le n\}$) with relations
\[
\theta_{i j} \theta_{k l} + \theta_{k l} \theta_{i j}=0, \quad
\text{for all \ } i, j, k ,l.
\]
Given an $mn$ tuple
$\epsilon=(\epsilon_{1 1}, \dots, \epsilon_{1 n}, \epsilon_{21},
\dots, \epsilon_{2n}, \dots, \epsilon_{m1}, \dots, \epsilon_{nm})$
with $\epsilon_{i j}=0$ or $1$, write
\[
\theta(\epsilon) = (\theta_{1 1})^{\epsilon_{1 1}} \dots (\theta_{1
n})^{\epsilon_{1 n}} (\theta_{2 1})^{\epsilon_{21}} \dots (\theta_{2
n})^{\epsilon_{2 n}} \dots (\theta_{m 1})^{\epsilon_{m 1}} \dots
(\theta_{m n})^{\epsilon_{m n}}.
\]
It is well known that the elements $\theta(\epsilon)$ form a
basis of $\cK[\theta]$.

Define an action of $\Lambda_q(V)$ on $\cK[\theta]$ by recursion
on degree as follows.
First we specify the action on lower degree subspaces in the
following way.
\begin{itemize}
\item Let
\[ X_{i k}(1) = \theta_{i k}, \quad \forall i, k;
\]
\item then
define the action of $X_{j k}$ on $\theta(\epsilon)$ as
follows:
\begin{enumerate}
\item If for all $a<j$ and $t>k$, we have $\epsilon_{a t}=0$, then
\[ X_{j k} \theta(\epsilon) = (-1)^{\sum_{a=1}^{j}\sum_{t=1}^k \epsilon_{a
t}} q^{-\sum_{a=1}^{j-1} \epsilon_{a k}}\theta(\epsilon+\1_{j k}),
\]
where $\1_{i k}$ denotes the $mn$ tuple whose $(i, k)$ entry
is $1$ and all of whose other entries are zero.
\item Otherwise, let $i<j$ be the smallest integer such that there
exists some $t>k$ with $\epsilon_{i t}=1$, and let $l$ be the
smallest such $t$. Let
\[
\begin{aligned}
\epsilon_{<}&= (\epsilon_{1 1}, \dots, \epsilon_{1 n}, \epsilon_{2
1}, \dots \epsilon_{2 n}, \dots, \epsilon_{i 1}, \dots, \epsilon_{i
l}, 0, \dots,
0), \\
 \epsilon_>&= \epsilon-\epsilon_<.
\end{aligned}
\]
Define
\[
\begin{aligned}
X_{j k} \theta(\epsilon) &= (-1)^{\sum_{a=1}^{i}\sum_{t=1}^k
\epsilon_{a t}} q^{-\sum_{a=1}^{i} \epsilon_{a k}}
[\theta(\epsilon_<)
(X_{j k}\theta(\epsilon_>)) \\
&+ (q-q^{-1}) \theta(\epsilon_<+\1n_{i k}-\1_{i l})(X_{j
l}\theta(\epsilon_>) )].
\end{aligned}
\]
Note that the degree of $\theta(\epsilon_>)$ is smaller than that of
$\theta(\epsilon)$, so that by induction, the action of $X_{j k}$ and
$X_{j l}$ on it are already defined. Therefore $X_{j
k}\theta(\epsilon_>)$ and $X_{j l}\theta(\epsilon_>)$ are well
defined elements in $\cK[\theta]$.
\end{enumerate}
\end{itemize}
The construction guarantees that the defining relations of
$\Lambda_q(V)$ are satisfied. But clearly
\[
X(\ep)(1) = \theta(\epsilon),
\]
and the linear independence of the elements $X(\ep)$ follows from
that of the $\theta(\epsilon)$ in the
Grassmann algebra $\cK[\theta]$.
\end{proof}

The following result is a generalisation to the quantum group
setting of \cite[Theorem 4.1.1]{H} known as `skew $(GL_m, GL_n)$
duality' in the terminology of Howe.
\begin{theorem}\label{skewduality}
Let $V^{(m)}$ and $V^{(n)}$ be the natural modules for
$\Uq(\mathfrak{gl}_m)$ and $\Uq(\mathfrak{gl}_n)$ respectively, and
set $V=V^{(m)}\otimes V^{(n)}$. Then as a
$\Uq(\mathfrak{gl}_m)\otimes\Uq(\mathfrak{gl}_n)$-module, the
braided exterior algebra $\Lambda_q(V)$ decomposes into a
multiplicity free direct sum of irreducibles as follows.
\begin{eqnarray*}
\Lambda_q(V) = \bigoplus_{\lambda} L^{(m)}_\lambda\otimes
L^{(n)}_{\lambda'},
\end{eqnarray*}
where $\lambda'$ denotes the conjugate of the partition $\lambda$,
and the sum is over all partitions $\lambda=(\lambda_1,
\lambda_2, \dots, \lambda_m)$ such that $\lambda_1\le n$.
\end{theorem}

The proof of the theorem will make use of the following result.
\begin{lemma}\label{h.w.-vectors}
For any partition $\lambda=(\lambda_1, \lambda_2, \dots,
\lambda_m)$ with $\lambda_1\le n$, let
\[
\Pi_\lambda:=X_{1 1} \dots X_{1 \lambda_1} X_{2 1}  \dots X_{2
\lambda_2}\dots X_{m 1}  \dots X_{m \lambda_m} \in \Lambda_q(V).
\]
Then
\begin{enumerate}
\item The element $\Pi_\lambda$ is  a highest weight
vector with respect to the actions of $\Uq(\mathfrak{gl}_m)$ and of
$\Uq(\mathfrak{gl}_n)$.
\item The $\Uq(\mathfrak{gl}_m)$
weight of $\Pi_\lambda$ is $\lambda$, while the $\Uq(\mathfrak{gl}_n)$
weight is $\lambda'$.
\end{enumerate}
\end{lemma}
\begin{proof}
It is understood that if $\lambda_i=0$, then $\Pi_\lambda$ contains
no $X_{i t}$ as a factor.  This also implies that there is no $X_{j
t}$ factor in $\Pi_\lambda$ for all $j\ge i$.

If $\lambda_{i+1}=0$, then the Chevalley generator
$e_i\in\Uq(\gl_m)$ acts on $\Pi_\lambda$ as $0$. Assume
$\lambda_{i+1}>0$. Then $e_i\Pi_\lambda$ is a linear combination of
terms of the form
\[
\Pi_\nu X_{i 1} \dots X_{i \lambda_i} X_{i+1, 1} \dots X_{i+1, t-1}
X_{i t} X_{i+1, t+1}\dots X_{i+1, \lambda_{i+1}} \dots  X_{m 1}
\dots X_{m \lambda_m}
\]
where $\nu=(\lambda_1, \dots, \lambda_{i-1}, 0, \dots, 0)$ and $t\le
\lambda_{i+1}$. We prove that this is equal to zero by proving
the more general result that for all $t$ such that $\lambda_i\ge t
> s$,
\begin{eqnarray}\label{left-maximal}
X_{i 1}X_{i 2} \dots X_{i \lambda_i} X_{i+1, 1}X_{i+1, 2} \dots
X_{i+1, s} X_{i, t}=0.
\end{eqnarray}
If $s=0$, then using the second relation of \eqref{relations-ss2} we
can shift $X_{i t}$ to the left to arrive at the expression
\[
(-q)^{\lambda_i-t} X_{i 1}X_{i 2}\dots X_{i t} X_{i t} X_{i, t+1}
\dots X_{i \lambda_i},
 \]
which is zero by the first relation of
\eqref{relations-ss2}. We now use induction on $s$. By
\eqref{relations-aass}, we may write
$
X_{i+1, s} X_{i, t}= -  X_{i, t} X_{i+1, s} +(q-q^{-1}) X_{i s}
X_{i+1, t}.
$
Thus
\[
\begin{aligned}
&X_{i 1}X_{i 2} \dots X_{i \lambda_i} X_{i+1, 1}X_{i+1, 2} \dots
X_{i+1, s} X_{i, t} \\
&= X_{i 1}X_{i 2} \dots X_{i \lambda_i} X_{i+1, 1}X_{i+1, 2} \dots
X_{i+1, s-1} (-  X_{i, t} X_{i+1, s} +(q-q^{-1}) X_{i s} X_{i+1,
t}),
\end{aligned}
\]
from which  \eqref{left-maximal} follows by the induction
hypothesis.

The fact that $\Pi_\lambda$ is also a $\Uq(\gl_n)$ highest weight
vector is proved similarly, by applying the Chevalley generators
$e_i\in\Uq(\gl_n)$ to $\Pi_\lambda$ and using the first and second
relations of \eqref{relations-ss2}.

The second part of the lemma is obvious.
\end{proof}

\begin{proof}[Proof of Theorem \ref{skewduality}]
For any $\Z_{\geq 0}$-graded module $M$, let $M_{\leq k}$ denote the 
sum of the graded components of degree at most $k$.
It follows from Lemma \ref{h.w.-vectors} that $\bigoplus_{\lambda}
L^{(m)}_\lambda\otimes L^{(n)}_{\lambda'}$ is a $\Uq(\fg)$-submodule
of $\Lambda_q(V)$, where $\fg=\gl_m\times\gl_n$. Let
$\Lambda(\C^m\otimes\C^n)$ be the exterior algebra of
$\C^m\otimes\C^n$ over $\C$, and
denote by $|\lambda|$ the size of the partition $\lambda$, that
is, the sum of its parts (recall we always assume that
$\lambda_1\leq n$). Then by skew
$(GL_m, GL_n)$ duality \cite[Theorem 4.1.1]{H}, we have
\[
\sum_{|\lambda|\le k} \dim_{\cK}\left(L^{(m)}_\lambda\otimes
L^{(n)}_{\lambda'} \right) = \dim_\C \Lambda(\C^m\otimes\C^n)_{\le k}
\]
since irreducible modules of $\fg$ and of $\Uq(\fg)$ with the same highest weight
have the same dimension.  By
Proposition \ref{Lambda-gl}, $\dim_\C\Lambda(\C^m\otimes\C^n)_{\le k} =
\dim_{\cK} \Lambda_q(V)_{\le k}$ for all $k$. Hence
\[
\dim_{\cK} \Lambda_q(V)_{\le k} = \sum_{|\lambda|\le k}
\dim_{\cK}\left(L^{(m)}_\lambda\otimes L^{(n)}_{\lambda'} \right),
\quad \forall k.
\]
This completes the proof.
\end{proof}

\section*{Acknowledgement}
The authors thank the Australian Research Council
and National Science Foundation of China for their financial support.

\end{document}